\documentclass[11pt]{article}
\usepackage{amsfonts,amsmath,amsopn,amssymb,amstext,amsthm,rotating}

\usepackage{mathtools}

\usepackage{enumerate}   
\usepackage{xcolor}

\usepackage{subcaption}
\usepackage{graphicx}

\usepackage{geometry}
\usepackage{shuffle}

\usepackage{url}
\usepackage{subcaption}

\usepackage[colorlinks,linkcolor=myGreen,citecolor=myGreen,urlcolor=black]{hyperref}
\usepackage[nameinlink,noabbrev]{cleveref} 
\usepackage{thmtools}

\usepackage[colorinlistoftodos, bordercolor=black, backgroundcolor=orange!20, linecolor=orange, textsize=scriptsize]{todonotes}

\usepackage[utf8]{inputenc} 

\usepackage[backend=biber,style=alphabetic,maxnames=6]{biblatex}
\AtEveryBibitem{
   \clearfield{month}
   \clearfield{series}
   \clearfield{venue}
   \clearname{editor}
   \clearlist{location} 
   \clearfield{doi}
   \clearfield{url}
   \clearfield{eprint}
   \clearfield{venue}
   \clearfield{issn}
   \clearfield{isbn}
   \clearfield{urldate}
   \clearfield{eventdate}
   \clearfield{pages}
   \clearfield{number}
   \clearfield{volume}
   \clearfield{urldate}
   \clearfield{urlmonth}
   \clearfield{urlyear}

}

\newtheorem{theorem}{Theorem}[section]

\newtheorem{conjecture}[theorem]{Conjecture}
\newtheorem{corollary}[theorem]{Corollary}

\newtheorem{example}[theorem]{Example}

\newtheorem{lemma}[theorem]{Lemma}
\newtheorem{definition}[theorem]{Definition}

\newtheorem{proposition}[theorem]{Proposition}
\newtheorem{remark}[theorem]{Remark}

\definecolor{blueXXredX}{rgb}{0.66, 0.33, 1.0}
\definecolor{blueXredXXXX}{rgb}{1.0, 0.22, 1.0}
\definecolor{darkBlue}{rgb}{0.16, 0.32, 0.75}
\definecolor{myGreen}{rgb}{0.18039216 0.49803922 0.09411765}
\definecolor{darkred}{rgb}{0.65, 0.0, 0.0}
\definecolor{CeruleanBlue}{rgb}{0.16, 0.32, 0.75}

\newcommand\w[1]{\mathbf{#1}}

\newcommand{\R}{\mathbb{R}}
\newcommand{\N}{\mathbb{N}}
\newcommand{\C}{\mathbb{C}}
\newcommand{\id}{\mathsf{id}}

\DeclareMathOperator{\len}{len}
\DeclareMathOperator{\weight}{weight}

\title{Universal Varieties for Signatures of Membranes}

\author{Carlos Am\'endola\\
  \small TU Berlin\\
  \small \texttt{amendola@math.tu-berlin.de}
  \and 
  Gabriel Riffo\\
  \small TU Berlin \\
  \small \texttt{riffo@tu-berlin.de}
  \and 
  Leonard Schmitz\\
   \small TU Berlin\\
    \small \texttt{lschmitz@math.tu-berlin.de}
 }

\date{}

\begin{document}

\maketitle

\begin{abstract}
Membranes are sufficiently smooth  vector-valued functions defined on the unit square, widely used in applications such as  interpolation of image data. The signature of a membrane is a collection of nonlinear features that capture its essential geometry. We index these features in the graded Hopf algebra of two-parameter shuffle products and study the algebraic varieties for each graded component. At each truncation level, we obtain the dimension of this variety via a novel Witt-type formula. We propose a Chen-Chow theorem for membranes, provably at the second truncation level, and conjecturally for higher levels with extensive computational evidence in the computer algebra system \texttt{OSCAR}. 
\end{abstract}

\\


\section{Introduction}

Chen’s iterated-integrals signature \cite{chen1954iterated} serves as a compact description of data streams and captures many of its important analytic and geometric properties.
After its original use in topology, it entered 
control theory \cite{fliess1976outil}, led to the development of rough path theory \cite{lyons1998differential,Friz_Victoir_2010}, and contributed to a wide range of applications in machine learning tasks on time series data \cite{kidger2020neural,kiraly2019kernels,chevyrev2025primersignaturemethodmachine}. 
In \cite{ibrahim2022imagesig}, Chen's  signature is applied for images, row-by-row, in a lightweight method for image recognition. 
A first adaptation of certain iterated integrals to interpolated image data (or membranes) is proposed by \cite{zhang2022two}, with applications to texture classification tasks. 

The first principal extension of signature-type objects from paths $X:[0,1]\to \R^d$ to membranes $X:[0,1]^2\to\R^d$ is the \emph{mapping space signature} introduced by \cite{GIUSTI2025102787}, based on Chen's original use in topology. The \emph{two-parameter sums signature} \cite{DS22} is a discrete adaptation from a combinatorial perspective, based on square increments, and characterizing equivalence of image data up to dynamic warping in horizontal and vertical direction. 
These two approaches are combined in \cite[Definitions 4.1 and 4.13]{diehl2025signature}, providing a notion of signature close to the original motivation in rough path theory and applications. This notion covers the integrals used in \cite{zhang2022two} and more recently, in  \cite{xie20262dsig,diehl2026tensor} applied to a semi-supervised framework in anomaly detection and classification on image data. In this work, we follow this notion of signature for membranes, which we now define. 

A \emph{membrane} is a smooth enough map $X:[0,1]^2\to\R^d$ from the unit square to $d$-dimensional such that the integrals \eqref{align:defDEFHTsig} below are defined. In particular, the mixed partial differentials $\partial_{12}X_j(s,t):=\tfrac{\partial}{\partial s}\tfrac{\partial}{\partial t}X_j(s,t)$ 
obey the usual rules of calculus. The signature lives in some infinite-dimensional graded space and associates with each membrane  a collection of integrals organized as entries of $k$-tensors of shape $d\times \dots \times d$. In contrast to path signatures, we attach to each of such tensors a permutation in the symmetric group $\Sigma_k$, permuting the order of integration as follows.

\begin{definition}[Membrane signature]
\label{def:signature_two_parameter}%
\label{def:2-param-signature}
Let $X:[0,1]^2\to\R^d$ be a membrane. 
For any $k\geq1$ and $\tau\in\Sigma_k$ we define the $k$-tensor
  $\sigma^{(\tau)}(X)\in(\R^d)^{\otimes k}$ with entries 
  \begin{align}\label{align:defDEFHTsig}
    \sigma^{(\tau)}_{i_1\dots i_k}(X) 
    &:=
    \int\limits_{\substack{0\leq s_1\leq\dots\leq s_k\leq 1\\ 0\leq t_1\leq\dots\leq t_k\leq 1}}
    \partial_{12} X_{i_1}(s_1,t_{\tau(1)})\dots\partial_{12} X_{i_k}(s_k,t_{\tau(k)})
    \;\mathrm{d}s_1\mathrm{d}t_1\dots\mathrm{d}s_k\mathrm{d}t_k
  \end{align} 
for $1\leq i_1,\dots,i_k\leq d$. For $k=0$, we set $\sigma^{(\emptyset)}(X)=1$.

The \emph{signature} of $X$ is the following  infinite sequence of tensors, 
\begin{equation}\label{eq:defSig}
    \sigma (X)\;:=\bigoplus_{k\in \N}\;\bigoplus_{\tau\in\Sigma_k}\sigma^{(\tau)}(X). 
\end{equation}
\end{definition}

When replacing our mixed partials $\partial_{12}$ with \emph{Jacobian minors}, we obtain the mapping space signature introduced in \cite[Definition 4.5]{GIUSTI2025102787}. 
When approximating the integrals in \eqref{align:defDEFHTsig} with sums and mixed partials with the corresponding difference operators, we obtain the two-parameter sums signature introduced in \cite[Definition 2.6]{DS22}. Compare also \cite[Example 4.16]{diehl2025signature} which discusses this connection in more detail.

Fixing all identity permutations $\id_k\in\Sigma_k$ in \eqref{align:defDEFHTsig}, the resulting subsequence of tensors 
\begin{equation}\label{eq:defidSig}\bigoplus_{k\in \N}\sigma^{(\id_k)}(X)\in\bigoplus_{k\in \N}(\R^d)^{\otimes k}\end{equation}
is called the \emph{id-signature of a membrane}. The second entry $\sigma^{(\id_2)}$ of this resulting subsequence \eqref{eq:defidSig} is a  matrix of size $d\times d$ studied in \cite{LS24} using tools from the theory of matrix congruence normal forms \cite{HORN20061010}.

The id-signature does not satisfy shuffle relations in the sense of  the shuffle identity for path signatures \cite[Theorem 2.14]{chevyrev2025primersignaturemethodmachine}. The shuffle product implies that the product of two terms of the path signature can be expressed as a linear combination of higher order terms. This fact has important theoretical and practical consequences. For id-signatures of membranes this is not possible, which is relevant for practical applications, e.g.,  \cite[Section 4.1]{xie20262dsig} discusses this limitation, and \cite[Table A.12]{diehl2026tensor} provides concrete experiments where the transition from the id-signature to its full membrane signature is crucial.
For this reason, the membrane signatures in \eqref{align:defDEFHTsig} carry additional permutations.

In this work, we study the membrane signature (\Cref{def:signature_two_parameter}) from an applied algebraic geometry perspective, in particular, developing a theory on algebraic varieties inspired by \cite{bib:AFS2019} on path signatures.
For a notion of ambient space,  we determine the correct Hopf algebra (\Cref{thm:HopfAlgerbra}) in spirit of \cite{DS22} and index the signature coefficients of a membrane via a character from the resulting free two-parameter shuffle algebra (\Cref{cor:free}). 
This algebra is graded, and we construct a new class of relations (\Cref{cor:shuffle-relations}) for each graded component.

We focus on two natural families of membranes, namely, polynomial and piecewise bilinear membranes. Taking a graded component of the signature as a map restricted to those two families, it becomes polynomial and admits dictionaries (\Cref{twoparamProdMemb}) inspired by \cite{bib:PSS2019}. The resulting images are semialgebraic and lie in the universal variety of membranes, which cuts out projected two-parameter shuffle relations.  
We develop an algebraic Chen-Chow theorem (\Cref{cor:algChenChow-k2}) for the $2$nd graded component, stating that the $2$-level universal variety can be expressed as the projected signature variety of piecewise bilinear membranes. 
Furthermore, we determine the dimension of the $k$-level universal variety, by providing a novel Witt-type formula (\Cref{thm:witt-two-parameter}), and thereby counting the number of free Lyndon generators from \cite{schmitz2024freegeneratorshoffmansisomorphism} in closed form.

For actual computations, we extend the methods in \cite{RS26} from paths to membranes, making signature tensors of membranes accessible in the free and open source computer algebra system \texttt{OSCAR} \cite{OSCAR-book}. We provide several illustrative examples and, in particular, computational evidence for \Cref{con:algebraicChenChowAllTruncLevels} which lifts the algebraic Chen-Chow theorem for membranes to the $k$-level universal variety.

The rest of the paper is organized as follows. In \Cref{sec:2paramHopf} we start with the algebraic underpinning of signatures, including two-parameter shuffle products and an indexing Hopf algebra that is compatible with a weak notion of Chen's identity. \Cref{sec:symOrbSum} focuses on a fixed graded component and introduces symmetrized orbit sums for relations within graded components. In \Cref{sec:Varieties} we construct algebraic varieties from these graded components, the main objects referred to in the title of this work. \Cref{sec:ChenChow} presents an algebraic Chen-Chow theorem for the second graded component, before moving to our Witt-type formula for every level in \Cref{sec:Witt}. 

\section{Hopf algebras of two-parameter shuffles}\label{sec:2paramHopf}
In this section, we recall the two-parameter shuffle algebra  \cite[Theorem 2.3]{schmitz2024freegeneratorshoffmansisomorphism}. For a general introduction to Hopf algebras we refer to \cite{cartier2021classical,grinbergreiner2020hopf}. For an indexing space of membrane signatures \eqref{eq:defSig} it suffices to consider only a certain sub-algebra constructed as follows.  For fixed $d\in\N$ let $\{\w{1},\dots,\w{d}\}^*$ denote the free monoid generated by the $d$ symbols $\w{1},\dots,\w{d}$. Let $\len(\cdot)$ denote the length homomorphism, uniquely determined by $\len(\w{1})=\dots=\len(\w{d})=1$. 
In this work we use the one-line notation for permutations, i.e., we identify $\tau\in\Sigma_k$ with the vector of natural numbers $(\tau(1)\,\dots\,\tau(k))$. 
We define the monoid $(C,\searrow,\mathrm{e})$ over the following set of pairs
\begin{equation*}C:=\{(w,\tau)\mid w\in\{\w{1},\dots,\w{d}\}^*\text{ and }\tau\in\Sigma_{\len(w)}\}
\end{equation*}
where we denote by $\mathrm{e}$ the pair of the empty word and the permutation of the empty set. 
The binary operation $\searrow:C\times C\to C$ is given by
$$(w_1,\tau_1)\searrow(w_2,\tau_2):=(w_1w_2,\tau_1\triangleright\tau_2).$$
Here $\tau_1\triangleright\tau_2:=(\tau_1(1)\dots \tau_1(k_1)\;\tau_2(1)+k_1\,\dots\, \tau_2(k_2)+k_1)$  denotes the \emph{shifted concatenation of permutations} for $\tau_i\in\Sigma_{k_i}$.  
 
\begin{remark}\label{rem:matrixNotation}
Note that $(C,\searrow,\mathrm{e})$ can be considered as a submonoid of all matrix compositions according to \cite[Definition 2.1]{DS22}.
For this we identify the pair $(w,\tau)$ by the matrix resulting from the diagonal matrix with the letters of $w$ on its diagonal, and permuting its rows according to $\tau$.  The symbol $\searrow$ is inspired by \cite[Proposition 3.2]{foissy2025matrixsymmetricquasisymmetricfunctions}, where our pairs relate to packed integer matrices due to the restriction $d=1$. 
\end{remark}

\begin{example}
We have $(\w{12},2\,1)\searrow (\w{324},1\,3\,2)=(\w{12324},2\,1\,3\,5\,4)$. 
\end{example}

We denote by $C_\mathsf{irr}\subseteq C$ the subset of \emph{irreducible pairs}, i.e., all $(w,\tau)\in C\setminus\{\mathrm{e}\}$ 
which have only trivial decompositions, i.e.,  for $(w,\tau)=(w_1,\tau_1)\searrow (w_2,\tau_2)$ follows $(w_1,\tau_1) = \mathrm{e}$ or $(w_2,\tau_2) = \mathrm{e}$. 
\begin{example}
We have $(\w{12},2\,1)\in C_\mathsf{irr}$ but $(\w{324},1\,3\,2)=(\w{3},1)\searrow(\w{24},2\,1)\not\in C_\mathsf{irr}$.
\end{example}

We denote by $H:=\R\langle C_\mathsf{irr}\rangle$ the free associative algebra over $C_\mathsf{irr}$ where its product extends the product $\searrow$ in $C$. 
With this we equip $H$ with the coproduct $\Delta$ that dualizes $\searrow$ according to \cite[Definition 2.10]{DS22}, see also \cite[Lemma 4.1 and Proposition 3.2]{foissy2025matrixsymmetricquasisymmetricfunctions}. Note that $C$ is an $\R$-vector space basis of $H$. 
\begin{definition}
    Let $\Delta:H\to H\otimes H$
be the coalgebra homomorphism that maps all $\R$-basis elements $(w,\tau)\in C$ to 
$$\Delta(w,\tau) := \sum_{(w_1,\tau_1)\searrow(w_2,\tau_2)=(w,\tau)}
(w_1,\tau_1)\otimes(w_2,\tau_2).$$
\end{definition}

\begin{example}
We have
\begin{align*}\Delta(\w{12324},2\,1\,3\,5\,4)
= \mathrm{e} \otimes (\w{12324},2\,1\,3\,5\,4)
&+(\w{12},2\,1)\otimes(\w{324},  1\,3\,2)\\
&+(\w{123},2\,1\,3 )\otimes (\w{24},2\,1)
+(\w{12324},2\,1\,3\,5\,4)\otimes \mathrm{e}. 
\end{align*}
\end{example}

We now equip $H$ with the two-parameter shuffle product \cite[Definition 2.1]{schmitz2024freegeneratorshoffmansisomorphism}, here presented in the notation using pairs in $C$. 
Let $\mathsf{Sh}(m,s)$ with $m,s\in\N$ denote the subset of \((m,s)\)-shuffles, that is, permutations $\nu\in\Sigma_{m+s}$ 
such that $\nu(1)<\dots<\nu(m)$ and $\nu(m+1)<\dots<\nu(m+s)$.

Furthermore, let
$\pi\in\Sigma_k$ act on word $w=w_1\dots w_k\in\{\w{1},\dots,\w{d}\}^*$ of length $k$ by permuting
positions according to $\pi$, that is 
\begin{equation}\label{eq:permActsOnWords}
\pi\cdot w:=(w_{\pi^{-1}(1)}\ldots w_{\pi^{-1}(k)}). 
\end{equation}

\begin{definition}\label{def:twoParamSh}
    The two-parameter shuffle product is the bilinear map
$\shuffle:H\times H\to H$
determined 
 by
$$
(w_1,\tau_1)\shuffle(w_2,\tau_2):=\sum_{\pi\in\mathsf{Sh}(k_1,k_2)}\ 
\sum_{\lambda\in\mathsf{Sh}(k_1,k_2)}
\bigl(\pi\cdot w_1w_2,\ \lambda\circ(\tau_1\triangleright\tau_2)^{-1}\circ\pi^{-1}\bigr)$$
for nonempty $(w_i,\tau_i)\in C$ with $k_i=\len(w_i)$, and with $\mathrm{e}$ as the neutral element for $\shuffle$. 
\end{definition}

\begin{example}
We have
\begin{align*}
(\w{12},1\,2)\shuffle(\w{3},1)
={} &(\w{123},1\,2\,3)+(\w{132},1\,3\,2)+(\w{312},3\,1\,2)\\
&+(\w{123},1\,3\,2)+(\w{132},1\,2\,3)+(\w{312},2\,1\,3)\\
&+(\w{123},2\,3\,1)+(\w{132},2\,1\,3)+(\w{312},1\,2\,3).
\end{align*}
\end{example}

\begin{theorem}\label{thm:HopfAlgerbra}
  $(H,\shuffle,\Delta)$ is a commutative, graded, and connected Hopf algebra.
\end{theorem}
\begin{proof}
Clearly $(H,\shuffle,\Delta)$ is a sub-Hopf algebra of \cite[Theorem 2.3]{schmitz2024freegeneratorshoffmansisomorphism} by restricting matrix compositions to $C$.  
In particular, we obtain a grading via the homomorphism of monoids $\weight(\cdot)$ determined by $\weight(w,\tau):=k$ for every $(w,\tau)\in C$ with $\len(w)=k$, see also \cite[Definition 3.1]{foissy2025matrixsymmetricquasisymmetricfunctions}. 
As an alternative proof, we can use the two-parameter version of Hoffman's isomorphism \cite[Theorem 4.1]{schmitz2024freegeneratorshoffmansisomorphism} and the original Hopf algebra of two-parameter quasi-shuffle products from \cite[Theorem 2.11]{DS22}.
\end{proof}

\begin{corollary}\label{cor:free}
    $(H,\shuffle,+)$ is a commutative, graded and free algebra. 
\end{corollary}
\begin{proof}
  With \Cref{thm:HopfAlgerbra}, freeness is a classical consequence of Leray's theorem, compare for instance \cite[Theorem 4.4.2]{cartier2021classical} or \cite[Theorem 1.7.29]{grinbergreiner2020hopf}.
\end{proof}

In \Cref{sec:Witt} we illustrate the resulting Lyndon-based free generators for $(H,\shuffle,+)$, as well as a novel Witt-type expression for its dimension. 

With the Hopf algebra $(H,\shuffle,\Delta)$ we can index the membrane signature (\Cref{def:signature_two_parameter}). For this we fix some membrane $X$ and identify $\sigma(X)$ by the homomorphism $\sigma:H\to \R$  of vector spaces that maps every nonempty $\R$-basis element $(w,\tau)\in C$ to \begin{equation}\label{eq:defNotSigIndexPair}
\sigma_{(w,\tau)}:=\sigma^{(\tau)}_w(X). 
\end{equation}The connected component of $H$ is mapped to $\sigma_{\mathrm{e}}:=1$. 
We recall the two-parameter shuffle identity for signatures; see also \cite[Theorem 1]{GIUSTI2025102787} for an analogous statement for mapping space signatures, \cite[Theorem 2.13]{DS22} for two-parameter sums signatures on discrete data, or \cite[Theorem 2.21]{horozov2015noncommutative} for type-\textbf{b} integrals of membranes. 
\begin{theorem}[Shuffle identity]\label{thm:shuffle-identity}
Given a membrane signature $\sigma:H\to \R$, we have $$\sigma_{(w_1,\tau_1)}\cdot\sigma_{(w_2,\tau_2)}=\sigma_{(w_1,\tau_1)\shuffle (w_2,\tau_2)}$$
   for $\R$-basis elements $(w_i,\tau_i)\in C$ of $H$.  
\end{theorem}
\begin{proof}
This follows from \cite[Proposition 4.29]{diehl2025signature} and the linear pairing \eqref{eq:defNotSigIndexPair}. 
\end{proof}

The shuffle identity turns $\sigma:H\to\R$ from \eqref{eq:defNotSigIndexPair} into a character. 
The Hopf algebra structure of $H$ is compatible with the following concatenation, illustrated by \Cref{fig:concatenationMembranes}. 

\begin{definition}\label{def:diagConcat}
For membranes $X,Y:[0,1]^2\to\R^d$, their \emph{diagonal concatenation} is the membrane $X\sqcup Y:[0,1]^2\to\R^d$ defined by
\[
(X\sqcup Y)(s,t):=\begin{cases}
X(2s,2t) & (s,t)\in[0,\tfrac{1}{2}]^2,\\[2pt]
X(1,1)+Y(2s-1,2t-1)-Y(0,0) & (s,t)\in[\tfrac{1}{2},1]^2,\\[2pt]
X(2s,1)+Y(0,2t-1)-Y(0,0) & (s,t)\in[0,\tfrac{1}{2}]\times[\tfrac{1}{2},1],\\[2pt]
X(1,2t)+Y(2s-1,0)-Y(0,0) & (s,t)\in[\tfrac{1}{2},1]\times[0,\tfrac{1}{2}].
\end{cases}
\]
\end{definition}

\begin{figure}
    \centering
    \begin{subfigure}{0.3\textwidth}
        \centering        \includegraphics[width=\textwidth]{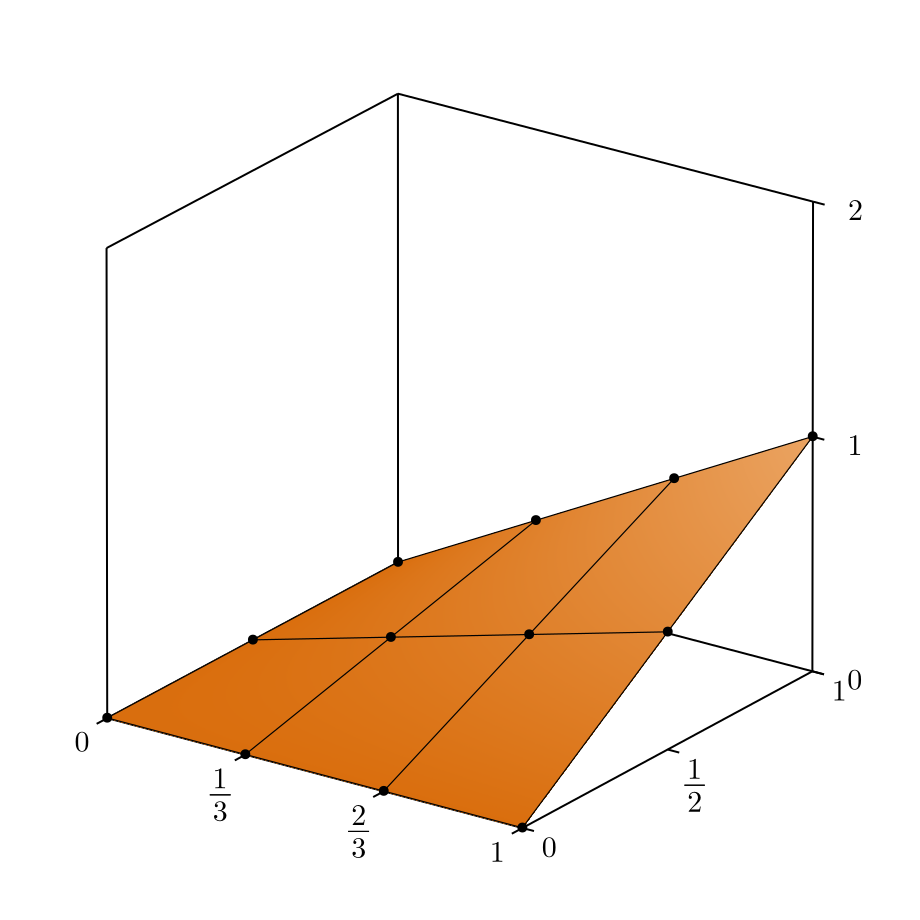}
        \caption*{$X(s,t)=st$.}
    \end{subfigure}
    \begin{subfigure}{0.3\textwidth}
        \centering
    \includegraphics[width=\textwidth]{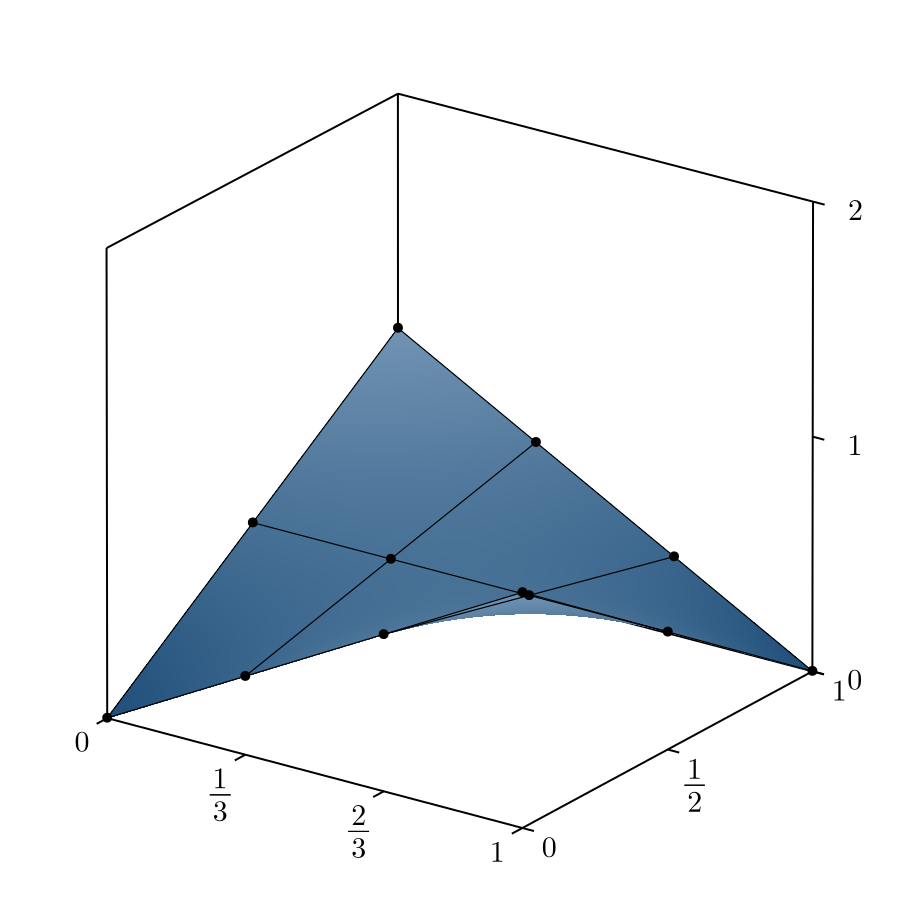}
        \caption*{$Y(s,t)=s+t-2st$.}
  
    \end{subfigure}
        \begin{subfigure}{0.3\textwidth}
        \centering
        \includegraphics[width=\textwidth]{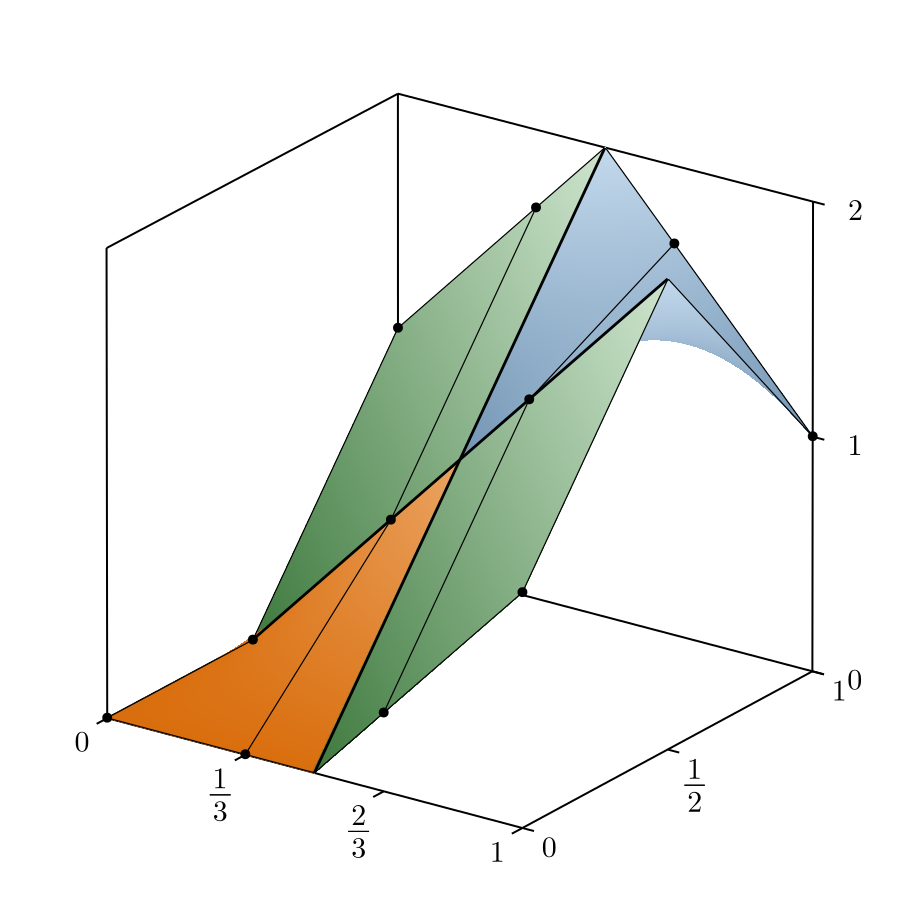}
        \caption*{$(X\sqcup Y)(s,t)$.}
    \end{subfigure}
    \caption{Diagonal concatenation of two membranes.}
    \label{fig:concatenationMembranes}
\end{figure}

Here, $X$ occupies the lower-left block $[0,\tfrac12]^2$ and $Y$ the upper-right block $[\tfrac12,1]^2$, while on each off-diagonal block $X\sqcup Y$ is a sum of functions that are constant in either $s$ or $t$. As a consequence, the mixed increment vanishes there,
\[
\partial_{12}(X\sqcup Y)=0 \quad\text{on }[0,\tfrac{1}{2}]\times[\tfrac{1}{2},1] \text{ and } [\tfrac{1}{2},1]\times[0,\tfrac{1}{2}],
\]
so the off-diagonal blocks become trivial after integration in the signature. On the two diagonal blocks,  the chain rule gives, respectively, $\partial_{12}(X\sqcup Y)=4\,\partial_{12}X(2s,2t)$ on $[0,\tfrac{1}{2}]^2$ and $4\,\partial_{12}Y(2s-1,2t-1)$ on $[\tfrac{1}{2},1]^2$. The correction term $-Y(0,0)$ ensures continuity at the corner $(\tfrac{1}{2},\tfrac{1}{2})$.
Inspired by \cite[Definition 2.29 and Corollary 2.31]{DS22} on discrete data, we can thus formulate the following Chen-type formula for membranes. 

\begin{theorem}[Weak Chen's identity]\label{thm:Chen}
For all membranes $X,Y:[0,1]^2\to \R^d$ and $(w,\tau)\in C$ we have 
$$\sigma^{(\tau)}_w(X\sqcup Y)=\sum_{(w_1,\tau_1)\searrow(w_2,\tau_2)=(w,\tau)}
\sigma^{(\tau_1)}_{w_1}(X)\cdot\sigma^{(\tau_2)}_{w_2}(Y).$$
\end{theorem}

\begin{proof}
Fix $(w,\tau)\in C$ with $k=\len(w)$. By \Cref{def:diagConcat}, $\partial_{12}(X\sqcup Y)$ vanishes on the two off-diagonal blocks. Hence, in the integral defining $\sigma^{(\tau)}_w(X\sqcup Y)$, the integrand
\begin{equation*}
\prod_{j=1}^{k}\partial_{12}(X\sqcup Y)_{w_j}(s_j,t_{\tau(j)})
\end{equation*}
is nonzero only when each $(s_j,t_{\tau(j)})$ lies either in $[0,\tfrac12]^2$ or in $[\tfrac12,1]^2$.

With $s_1\leq\dots\leq s_k$, there exists $0\leq k_1\leq k$ such that $s_j\leq\tfrac12\iff j\leq k_1$. Analogously, there is $0\leq\ell\leq k$ such that  $t_r\leq\tfrac12\iff r\leq\ell$. With this we obtain $\tau(\{1,\dots,k_1\})\subseteq\{1,\dots,\ell\}$ and $\tau(\{k_1+1,\dots,k\})\subseteq\{\ell+1,\dots,k\}$, hence $\ell=k_1$. For $k_2:=k-k_1$ there are $\tau_i\in\Sigma_{k_i}$ such that $\tau=\tau_1\triangleright\tau_2$. Decomposing $w=w_1w_2$ accordingly gives us $(w,\tau)=(w_1,\tau_1)\searrow(w_2,\tau_2)$.

We fix such a decomposition. If the first $k_1$ points lie in $[0,\tfrac12]^2$ and the last $k_2$ in $[\tfrac12,1]^2$, the constraints decouple across the two blocks, so the integral factorizes. We apply the change of variables $(s_j,t_r)\mapsto(2s_j,2t_r)$ in the first block and $(s_j,t_r)\mapsto(2s_j-1,2t_r-1)$ in the second, and use that $\partial_{12}(X\sqcup Y)$ equals $4\,\partial_{12}X(2s,2t)$ in $[0,\tfrac12]^2$ and $4\,\partial_{12}Y(2s-1,2t-1)$ in $[\tfrac12,1]^2$. Here, the Jacobians $4^{-k_1}$ and $4^{-k_2}$ cancel the chain-rule factors. 
The claim follows by summing over $k_1$. 
\end{proof}

The Chen-type formulas for vertical and horizontal concatenation from \cite[Theorem 8.5]{GIUSTI2025102787} or \cite[Proposition 4.8 and Remark 4.15]{diehl2025signature} only handle id-signatures \eqref{eq:defidSig} and rely on an additional convolution product. 
The advantage of our novel Chen-type formula (\Cref{thm:Chen}) is the formulation over the entire signature \eqref{eq:defSig}, paying the price that we restrict to a subclass of $\sqcup$-decomposable membranes.
In particular, the lack of a satisfying analogue of Chen's identity for any membrane on the entire signature has consequences for computational complexity.
In the case of path signatures, one obtains a linear complexity algorithm for piecewise linear paths by iteratively applying Chen's identity. Alternatively, an exponential algorithm based on symmetric Tucker decompositions is proposed in \cite[Section 2]{bib:PSS2019}, see also \cite[Table 1]{RS26} for benchmarks of regimes where this approach outperforms the Chen-based algorithm. 

For membranes, the computational complexity landscape of signatures is not fully understood. The first result in this direction is proposed in \cite[Theorem 4.5]{DS22}, a linear-time algorithm for two-parameter sums signatures at \emph{chained matrix compositions} on discretized membranes. In particular, this result covers the discretized notion of the id-signature  in linear time and appears in \cite[Appendix B]{xie20262dsig} to efficiently approximate certain integrals. 
In  \cite[Theorem 3.18]{LS24} a fast algorithm for id-signatures of piecewise bilinear membranes is developed, i.e., bilinear interpolations of data in $\R^d$ that is arranged on an $(m+1)\times (n+1)$
grid. 
Its complexity is $\mathcal{O}(d^kk^3mn)$ and the integrals are computed exactly. 
We implement this algorithm in \cite{RS26} and thus make it available in \texttt{OSCAR}.  
Finally, \cite[Theorem 3]{diehl2026tensor} proposes an efficient algorithm to compute a substantial larger family of signature coefficients of the two-parameter sums signature via a corner tree approach. 

Our weak Chen's identity (\Cref{thm:Chen}) provides an algorithm for computing the signature of membranes $X_1\sqcup\dots \sqcup X_m$ with known signatures $\sigma(X_i)$ that has linear-time complexity in $m$. In \Cref{sec:Varieties} we develop the notion of a dictionary for membrane signatures, which leads to an exponential-time algorithm for full signatures and any piecewise bilinear or polynomial membrane, based on Tucker decompositions. In the following, we propose a quadratic-time algorithm for id-signatures of membranes whose component functions are polynomials.

\begin{proposition}\label{prop:complexityPolynomialMembranes}
The $k$-tensor $\sigma^{(
\id_k)}(X)$ from the id-signature of a membrane $X$  with $d$ polynomial components of multidegree $(m,n)$ is computable in 
$\mathcal{O}(d^k k^3 m^2 n^2)$  operations. 
\end{proposition}

\begin{proof}
Fix indices $(i_1,\dots, i_k)\in\{1,\dots,d\}^k$  and, for $1\le r\le k$, set 
\[
P_r(s,t)\;:=\int\limits_{\substack{0\le u_1\le\dots\le u_r\le s\\ 0\le v_1\le\dots\le v_r\le t}}
\prod_{j=1}^r\partial_{12}X_{i_j}(u_j,v_j)\,\mathrm{d}u_1\mathrm{d}v_1\dots\mathrm{d}u_r\mathrm{d}v_r,
\]
so that $\sigma^{(\id_k)}_{i_1\dots i_k}(X)=P_k(1,1)$ by \eqref{align:defDEFHTsig}. Splitting off the outermost pair of integrations yields the recursion
\begin{equation}\label{eq:P-recursion}
P_r(s,t)\;=\;\int_0^s\!\int_0^t P_{r-1}(u,v)\,\partial_{12}X_{i_r}(u,v)\,\mathrm{d}v\,\mathrm{d}u
\qquad(2\le r\le k),
\end{equation}
while $P_1(s,t)=\int_0^s\!\int_0^t\partial_{12}X_{i_1}(u,v)\,\mathrm{d}v\,\mathrm{d}u=X_{i_1}(s,t)$, since $X_{i_1}$ vanishes on both coordinate axes. Iterating $k-1$ times the polynomial product $P_{r-1}\cdot\partial_{12}X_{i_r}$ followed by termwise integral in $s$ and $t$ produces $P_k$, and evaluation at $(1,1)$ returns $\sigma_{(w,\id_k)}$.

For the complexity, an induction on \eqref{eq:P-recursion} shows that $P_r$ has multidegree $(rm,rn)$.  Multiplication by $\partial_{12}X_{i_r}$, of multidegree $(m-1,n-1)$, gives multidegree $(rm-1,rn-1)$, and the integral restores $(rm,rn)$. Iteration $r$ therefore costs $\mathcal{O}(r^2m^2n^2)$ scalar operations for the product, dominating the $\mathcal{O}(r^2mn)$ divisions of the integral. Summation over $r=2,\dots,k$ yields $\mathcal{O}(k^3m^2n^2)$ operations per word, and there are $d^k$ words of length $k$. The precomputation of $\partial_{12}X_i$ for $i=1,\dots,d$ contributes $\mathcal{O}(dmn)$ and is absorbed into the bound $\mathcal{O}(d^kk^3m^2n^2)$. For a step-by-step description of the resulting algorithm we refer to our implementation in the package \cite{RS26}.  
\end{proof}

We close this section with a specific justification on the construction of our indexing Hopf algebra $H$.   
For this, note that only one of two simplices in \Cref{def:2-param-signature} is permuted. 
We could attach two permutations, one to each integral, but this leads to redundant linear relations, specified in the following. 

\begin{lemma}\label{prop:onePermSuff}
Let $\tau,\pi\in\Sigma_k$ and $1\leq i_1,\dots,i_k\leq d$. Then, 
\begin{equation*}
\int\limits_{\substack{s_{\tau(1)}\leq\dots\leq s_{\tau(k)}\\ t_{\pi(1)}\leq\dots\leq t_{\pi(k)}}}
\partial_{12} X_{i_1}(s_1,t_1)\dots\partial_{12} X_{i_k}(s_k,t_k)
\;\mathrm{d}s_1\mathrm{d}t_1\dots\mathrm{d}s_k\mathrm{d}t_k
\;=\;
\sigma^{(\pi^{-1}\circ\tau)}_{a_1\dots a_k}
\end{equation*}
where $a_j := i_{\tau(j)}$ for $1\leq j\leq k$.
\end{lemma}
 
\begin{proof}
Consider the change of variables
$
u_j = s_{\tau(j)}$ and $v_j = t_{\pi(j)}$ with $1\leq j \leq k$, 
which is a permutation of the coordinates of $[0,1]^{2k}$ and hence measure
preserving. Under it the domain of integration becomes the standard pair of
ordered simplices
$u_1 \le \dots \le u_k$ and $v_1 \le \dots \le v_k$. 
Inverting the substitution gives $s_r = u_{\tau^{-1}(r)}$ and
$t_r = v_{\pi^{-1}(r)}$, so the $r$th factor of the integrand reads
$\partial_{12}X_{i_r}\bigl(u_{\tau^{-1}(r)},v_{\pi^{-1}(r)}\bigr)$. Reordering the
factors according to the index of the first argument, i.e.,  substituting
$r=\tau(j)$, the integrand becomes 
\begin{equation*}
\prod_{j=1}^{k}\partial_{12}X_{i_{\tau(j)}}\bigl(u_j,\,v_{\pi^{-1}\circ\tau(j)}\bigr)
\;=\;\prod_{j=1}^{k}\partial_{12}X_{a_j}\bigl(u_j,\,v_{\rho(j)}\bigr),
\qquad \rho:=\pi^{-1}\circ\tau .
\end{equation*} 
By \Cref{def:2-param-signature} the resulting integral over the standard ordered
domain is $\sigma^{(\tau)}_{a_1\dots a_k}$.
\end{proof}
 
In particular, every pair $(\tau,\pi)$ in \Cref{prop:onePermSuff} enters only through the single permutation $\pi^{-1}\circ\tau$.
Permuting both simplices therefore produces no entries beyond those already
recorded. Hence, it suffices to consider only one permutation in our Hopf algebra $H$. 

\section{Symmetrized orbit sums}\label{sec:symOrbSum}
From this section onward, we assume $d\geq 2$. 
We next show that the 
 shuffle identity yields a family of polynomial relations at each graded component of $H$. 
We use the  following construction.
For the group action \eqref{eq:permActsOnWords} and a word $w\in\{\w{1},\dots,\w{d}\}^*$ of length $k$, we denote by
$\mathrm{orb}(w):=\{\tau\cdot w\mid\tau\in\Sigma_k\}$
the \emph{orbit} of $w$.
For a partition $\alpha=(\alpha_1,\dots,\alpha_r)$ of length $r$ and an $r$-tuple $i=(i_1,\dots, i_r)$ of letters in $H$, let 
$i^{\alpha}:=i_1^{\alpha_1}\ldots
i_r^{\alpha_r}$
be the word in which each $i_c$ is repeated $\alpha_c$ times.
For any such $\alpha$ and $i$ we introduce the \emph{symmetrized orbit sum} by
\begin{equation*}
\Theta(\alpha,i)\;:=\;
\sum_{\tau\in\Sigma_k}\;\sum_{w\in\mathrm{orb}(i^{\alpha})} \sigma_{(w,\tau)}.
\end{equation*}
The trivial partition $\alpha=(k)$ yields 
$\Theta(k,i)=\sum_{\tau\in\Sigma_k}\sigma_{(i^k,\tau)}$. 

\begin{theorem}\label{cor:shuffle-relations}
Let $\alpha=(\alpha_1,\dots,\alpha_r)$ be a partition of $k$ with length $r\geq2$ and $i$ be an $r$-tuple of pairwise distinct letters in $H$. 
 Then
\begin{equation}\label{eq:master-shuffle-relation}
\Theta(\alpha,i)^{\,p_\alpha}
\;=\;\binom{k}{\alpha_1,\dots,\alpha_r}^{\,p_\alpha}\prod_{c=1}^{r} \Theta(k,i_c)^{p_\alpha\alpha_c/k}
\end{equation}
where the exponents $p_\alpha:=k/\gcd(\alpha_1,\dots,\alpha_r)$ and $p_\alpha\alpha_c/k$ are positive integers.
\end{theorem}

We illustrate \Cref{cor:shuffle-relations} for small $k$ before giving its
proof, making the exponents $p_\alpha$ and the multinomial factors explicit.

\begin{example}\label{ex:shuffle-k2}
For $k=2$ the only partition with $r\geq2$ is $\alpha=(1,1)$, so the theorem yields
\begin{align*}
\left(\sigma_{(\w{ij},1\,2)}+\sigma_{(\w{ij},2\,1)}+\sigma_{(\w{ji},1\,2)}+\sigma_{(\w{ji},2\,1)}\right)^2
&=\Theta((1,1),\w{ij})^2\\
&=4\,\Theta(2,\w{i})\Theta(2,\w{j})\\
&=4\,\left(\sigma_{(\w{ii},1\,2)}+\sigma_{(\w{ii},2\,1)}\right)\left(\sigma_{(\w{jj},1\,2)}+\sigma_{(\w{jj},2\,1)}\right).
\end{align*}
\end{example}

\begin{example}\label{ex:shuffle-k3}
For $k=3$ the partitions with $r\geq2$ are $(2,1)$ and $(1,1,1)$, hence $p_\alpha=3$. Explicitly, for pairwise distinct $\mathbf{i},\mathbf{j}$ and $\mathbf{l}$, we have
\begin{equation*}
\Theta(3,\w{i})=\sum_{\tau\in\Sigma_3}\sigma_{(\w{iii},\tau)},\qquad
\Theta((2,1),\w{il})=\sum_{\tau\in\Sigma_3} \sum_{w\in\{\w{iil},\w{ili},\w{lii}\}}\sigma_{(w,\tau)},
\end{equation*}
\begin{equation*}
\Theta((1,1,1),\w{ijl})=\sum_{\tau\in\Sigma_3} \sum_{w\in \mathrm{orb}(\w{ijl})}\sigma_{(w,\tau)}.
\end{equation*}
\Cref{cor:shuffle-relations} gives us for the two partitions, respectively, 
$\binom{3}{2,1}=3$ and $\binom{3}{1,1,1}=6$, thus
\begin{equation*}
\Theta((2,1),\w{il})^{3}=3^3\,\Theta(3,\w{i})^{2}\,\Theta(3,\w{l})
\quad 
\text{ and }
\quad
\Theta((1,1,1),\w{ijl})^{3}=6^3\,\Theta(3,\w{i})\Theta(3,\w{j})\Theta(3,\w{l}).
\end{equation*}
\end{example}

\begin{example}\label{rem:shuffle-k4k5}
At $k=4$ we have four partitions
$(3,1),(2,2),(2,1,1),(1,1,1,1)$, and \Cref{cor:shuffle-relations} yields the relations 
\begin{align*}
\Theta((3,1),\w{il})^{4}&=4^4\,\Theta(4,\w{i})^{3}\Theta(4,\w{l}), \\
\Theta((2,2),\w{ij})^{2}&=6^2\,\Theta(4,\w{i})\Theta(4,\w{j}),\\
\Theta((2,1,1),\w{ijl})^{4}&=12^4\,\Theta(4,\w{i})^{2}\Theta(4,\w{j})\Theta(4,\w{l}), \\
\Theta((1,1,1,1),\w{ijlm})^{4}&=24^4\,\Theta(4,\w{i})\Theta(4,\w{j})\Theta(4,\w{l})\Theta(4,\w{m}).
\end{align*} 
Here $\gcd(2,2)=2$ makes $p_{(2,2)}=2$, so for $(2,2)$ the exponent drops to a square. At $k=5$,
prime, every partition with $r\geq2$ has $p_\alpha=5$. Therefore, the exponents $p_\alpha$ in the theorem do not drop as for the cases $k=2$ and $k=3$.
\end{example}

The proof of \Cref{cor:shuffle-relations} is carried out entirely inside the
Hopf algebra $H$. Apart from \Cref{thm:shuffle-identity}, which transports
identities in $H$ to identities among signature coordinates, no property of
membranes enters; the relations \eqref{eq:master-shuffle-relation} are thus
consequences of the two-parameter shuffle product of \Cref{def:twoParamSh}
alone. The proof rests on two identities in $H$, both obtained by induction: the first computes a single shuffle power
$(\mathbf{j},1)^{\shuffle m}$, the second extends it to a product following the
structure of the partition $\alpha$.

\begin{lemma}\label{lem:diagonal-identity}
    For every exponent $m\in \N$ we have
    \begin{equation*}
        (\mathbf{j},1)^{\shuffle m}=m!\sum_{\tau\in\Sigma_m}(\mathbf{j}^m,\tau).
    \end{equation*}
\end{lemma}
\begin{proof}
We use induction on $m$. For $m=1$, both sides equal $(\mathbf{j},1)$, so the
claim holds with $m=1$.
For $m=2$, applying \Cref{def:twoParamSh} with $k_1=k_2=1$ gives
\begin{equation*}
(\mathbf{j},1)\shuffle(\mathbf{j},1)
=\sum_{\pi\in\mathsf{Sh}(1,1)}\sum_{\lambda\in\mathsf{Sh}(1,1)}
\bigl(\pi\cdot\mathbf{jj},\ \lambda\circ\pi^{-1}\bigr)
=2\sum_{\tau\in\Sigma_2}(\mathbf{jj},\tau),
\end{equation*}
since $\pi\cdot\mathbf{jj}=\mathbf{jj}$ for both $\pi\in\mathsf{Sh}(1,1)$, and
$\lambda\mapsto\lambda\circ\pi^{-1}$ is a bijection of $\Sigma_2$ for each fixed $\pi$.
For the inductive step, assume the claim
    \begin{equation*}
(\w{j},1)^{\shuffle m} = m!\sum_{\tau\in \Sigma_m}(\w{j}^m,\tau)
    \end{equation*}
is true for $m\geq 2$. Then, 
\begin{equation*}
    (\w{j}^m,\tau)\shuffle (\w{j},1) = \sum_{\pi \in \mathsf{Sh}(m,1) }\sum_{\lambda \in \mathsf{Sh}(m,1) }(\pi \cdot \w{j}^{m+1},\ \lambda \circ(\tau\triangleright 1)^{-1}\circ\pi^{-1}).
\end{equation*}
    Since all letters coincide, $\pi \cdot \w{j}^{m+1}=\w{j}^{m+1}$ for all $\pi$, so the first entry is constant. For the second entry, note that $\tau \triangleright 1$ fixes $m+1$, and that a shuffle $\lambda \in \mathsf{Sh}(m,1)$ is determined by the value $\lambda(m+1)$.  Since distinct shuffles give distinct values of $\lambda(m+1)\in \{1,\dots,m+1\}$, and $\tau$ ranges over all of $\Sigma_m$, the map 
    \begin{equation*}
        \mathsf{Sh}(m,1)\times \Sigma_m \to \Sigma_{m+1}, \qquad (\lambda,\tau) \mapsto \lambda \circ (\tau \triangleright1)^{-1},
    \end{equation*}
    is a bijection: any $\rho \in \Sigma_{m+1}$ is recovered by taking $\lambda$ to be the unique shuffle with $\lambda(m+1)=\rho(m+1)$ and consider $\tau:=\lambda^{-1}\circ\rho$ as an element in $\Sigma_m$. For each fixed $\pi\in\mathsf{Sh}(m,1)$, right-composition with $\pi^{-1}$ is itself a bijection of $\Sigma_{m+1}$, so summing over the $m+1$ shuffles $\pi$ gives 
    \begin{equation*}
    (\w{j}^m,\tau)\shuffle (\w{j},1) = \sum_{\pi \in \mathsf{Sh}(m,1) }\sum_{\lambda \in \mathsf{Sh}(m,1) }(\pi \cdot \w{j}^{m+1},\ \lambda \circ(\tau\triangleright 1)^{-1}\circ\pi^{-1})= (m+1) \sum_{\tau \in \Sigma_{m+1}}(\w{j}^{m+1},\tau).
\end{equation*}
Using the induction hypothesis we have
    \begin{align*}
        (\w{j},1)^{\shuffle (m+1)}
        = m!\sum_{\tau \in \Sigma_{m}}\bigl((\w{j}^m,\tau) \shuffle (\w{j},1)\bigr) 
        = (m+1)!\sum_{\tau \in \Sigma_{m+1}} (\w{j}^{m+1},\tau)
    \end{align*}
    which is the claim for $m+1$.
\end{proof}

\begin{lemma}\label{lem:orbit-identity}
Let $\alpha=(\alpha_1,\dots,\alpha_r)$ be a partition of $k$ and 
$i_1,\dots,i_r$ be pairwise distinct letters. Then
\begin{equation*}
(i_1,1)^{\shuffle\alpha_1}\shuffle\cdots\shuffle(i_r,1)^{\shuffle\alpha_r}
=k!\,\binom{k}{\alpha_1,\dots,\alpha_r}^{-1}\sum_{w\in\mathrm{orb}(i^\alpha)}\sum_{\tau\in\Sigma_k}(w,\tau). 
\end{equation*}
\end{lemma}
\begin{proof}
    We first show, by the same induction as in \Cref{lem:diagonal-identity}, that 
\begin{equation}\label{eq:arbitrary-letters}
(i_1,1)\shuffle\cdots\shuffle(i_k,1)
=\sum_{\pi\in\Sigma_k}\ \sum_{\tau\in\Sigma_k}(\pi\cdot i_1\ldots i_k,\ \tau).
\end{equation}
The case $k=1$ is trivial. Assume \eqref{eq:arbitrary-letters} for some $m\geq 1$ and multiply by $(i_{m+1},1)$. Applying \Cref{def:twoParamSh} with $k_1=m$, $k_2=1$ to a term $(\pi\cdot i_1\ldots i_m,\tau)$ gives 
\begin{equation*}
    \sum_{\omega \in \mathsf{Sh}(m,1)} \sum_{\lambda \in \mathsf{Sh}(m,1)} \bigl(\omega \circ (\pi\triangleright 1 )
 \cdot i_1\ldots i_{m+1},\ 
\lambda\circ(\tau\triangleright1)^{-1}\circ\omega^{-1} \bigr) \end{equation*}
where we used $(\pi\cdot i_1\ldots i_m)\,i_{m+1}=(\pi\triangleright1)\cdot
i_1\ldots i_{m+1}$. The map
\begin{equation}\label{eq:proofThmRelationsBijaction}
    \mathsf{Sh}(m,1) \times \Sigma_m \to \Sigma_{m+1},\qquad (\omega,\pi)\mapsto \omega\circ(\pi\triangleright 1).
\end{equation}
is a bijection, thus
\begin{equation*}
\{\omega\circ(\pi\triangleright1)\mid \omega\in\mathsf{Sh}(m,1),\ \pi\in\Sigma_m\}
=\Sigma_{m+1}.
\end{equation*}
For each fixed $\omega$ the same bijection \eqref{eq:proofThmRelationsBijaction} applied to  $(\lambda,\tau)$, followed by right-composition with $\omega^{-1}$ (which is a bijection of $\Sigma_{m+1}$), shows that 
\begin{equation*}
    \{\lambda\circ(\tau\triangleright1)^{-1}\circ\omega^{-1}\mid \lambda\in\mathsf{Sh}(m,1),\ 
\tau\in\Sigma_m\}=\Sigma_{m+1}. 
\end{equation*}
 Therefore, the two indices are independent and each pair in $\Sigma_{m+1}\times\Sigma_{m+1}$
occurs once, which closes the induction for \eqref{eq:arbitrary-letters}. 

We now move on to the word $i^{\alpha}$. For each $w\in \mathrm{orb}(i^{\alpha})$, the number of $\pi \in \Sigma_k$ with $\pi\cdot i^{\alpha}=w$ equals the number of ways to permute positions within each block of equal letters, namely $\alpha!:=\alpha_1!\dots\alpha_r!$, independent of $w$. Combining this with the cardinality
\begin{equation*}
    \bigl |\mathrm{orb}(i^{\alpha})\bigr|= \binom{k}{\alpha_1,\dots,\alpha_r},
\end{equation*}
we obtain
\begin{equation*}
   \alpha!=\binom{k}{\alpha_1,\dots,\alpha_r}^{-1}k!
\end{equation*}
 for every $w\in \mathrm{orb}(i^{\alpha})$. Grouping the sum over $\pi$ in \eqref{eq:arbitrary-letters} by $\pi \cdot i^{\alpha}$ gives the claim.
\end{proof}

\begin{proof}[Proof of \Cref{cor:shuffle-relations}]

Let $\sigma_j:=\sigma_{(\mathbf{j},1)}$ denote the $j$th coordinate of the first
graded component of $\sigma(X)$. By the shuffle identity
(\Cref{thm:shuffle-identity}) and \Cref{lem:diagonal-identity},
\begin{equation}\label{eq:diagonal-identity-cor}
\sigma_{i_c}^{\,k}
=\sigma_{(i_c,1)^{\shuffle k}}
=k!\,\Theta(k,i_c).
\end{equation}
Analogously, by \Cref{lem:orbit-identity},
\begin{equation}\label{eq:base-identity}
\prod_{c=1}^{r}\sigma_{i_c}^{\alpha_c}
=\sigma_{(i_1,1)^{\shuffle\alpha_1}\shuffle\cdots\shuffle(i_r,1)^{\shuffle\alpha_r}}
=\binom{k}{\alpha_1,\dots,\alpha_r}^{-1}k!\,\Theta(\alpha,i).
\end{equation}
Note that $p_\alpha$ is the least positive integer with $k\mid p_\alpha\alpha_c$ for
all $c$. Hence, raising to the power $p_\alpha$ and combining \eqref{eq:diagonal-identity-cor} with \eqref{eq:base-identity}
gives us
\begin{equation*}
\Bigl(\binom{k}{\alpha_1,\dots,\alpha_r}^{-1}k!\,\Theta(\alpha,i)\Bigr)^{p_\alpha}
=\prod_{c=1}^{r}\sigma_{i_c}^{\,p_\alpha\alpha_c}
=\prod_{c=1}^{r}\bigl(k!\,\Theta(k,i_c)\bigr)^{p_\alpha\alpha_c/k}
\end{equation*}
so dividing both sides by $(k!)^{p_\alpha}$ yields the claim. 
\end{proof}

\section{Varieties of signature tensors}\label{sec:Varieties}

In this section, following \cite{bib:AFS2019}, we work over the algebraically closed field $\C$ and introduce and relate the
signature varieties of two natural families of membranes. 
All of these live in a universal variety for membranes which we construct as follows. 

By collecting all signature tensors \eqref{align:defDEFHTsig} with permutations in the same $\Sigma_k$ we obtain 
\begin{equation}\label{eq:signature_tensor}
\sigma^{(k)}(X):=\bigoplus_{\tau\in\Sigma_k}\sigma^{(\tau)}(X)\in\C^{d\times\dots\times d\times k!}
\end{equation}
which is the $k$th graded component of $\sigma(X)$, with the grading introduced in \Cref{thm:HopfAlgerbra}. By enumerating $\Sigma_k$, we identify the \emph{$k$-level signature} $\sigma^{(k)}(X)$ by a $d\times\dots\times d\times k!$ tensor of order $k+1$.  
In the previous section we saw that the entries of these tensors satisfy certain relations (\Cref{cor:shuffle-relations}) that come from the two-parameter shuffle product, \Cref{thm:shuffle-identity}.  In this section, we formally construct all such relations from a fixed graded component \eqref{eq:signature_tensor}. 
Inspired by the construction \cite[Section 4.3]{bib:AFS2019} for paths, we consider 
the polynomial ring $\C[\sigma^{(\leq k)}]:=\C[\sigma_{(w,\tau)}\mid (w,\tau)\in C,\len(w)\leq k]$ and the ideal
$$G_{d,k}:=\langle \sigma_{(w_1,\tau_1)}\sigma_{(w_2,\tau_2)}-\sigma_{(w_1,\tau_1)\shuffle (w_2,\tau_2)}\mid \len(w_1)+\len(w_2)\leq k\rangle$$
with respect to the two-parameter shuffle product, \Cref{def:twoParamSh}. 
This ideal is not homogeneous with respect to the degree. Let $\mathcal{G}_{d,k}$ denote the algebraic variety  given by the vanishing set of $G_{d,k}$. Here, the ambient space has dimension $1+d+2d^2+\dots +k!d^k$. 

\begin{example}\label{ex:G23}
The ideal $G_{2,3}$ lives in the polynomial ring $\C[\sigma^{(\leq 3)}]$ with 
$58$ variables. It is generated by $19$ shuffle relations: three of  $\weight$ $2$, 
\begin{align*}
\sigma_{(\w{1},1)}^{\,2} &- 2\sigma_{(\w{11},1\,2)} - 2\sigma_{(\w{11},2\,1)},\\
\sigma_{(\w{1},1)}\sigma_{(\w{2},1)} &- \sigma_{(\w{12},1\,2)} - \sigma_{(\w{12},2\,1)} - \sigma_{(\w{21},1\,2)} - \sigma_{(\w{21},2\,1)},\\
\sigma_{(\w{2},1)}^{\,2} &- 2\sigma_{(\w{22},1\,2)} - 2\sigma_{(\w{22},2\,1)},
\end{align*}
and 16 relations of $\weight$ $3$, such as
\begin{equation*}
\sigma_{(\w{1},1)}\sigma_{(\w{11},1\,2)} - 3\sigma_{(\w{111},1\,2\,3)} - 2\sigma_{(\w{111},1\,3\,2)} -2\sigma_{(\w{111},2\,1\,3)}- \sigma_{(\w{111},2\,3\,1)} - \sigma_{(\w{111},3\,1\,2)}.
\end{equation*}
The variety $\mathcal{G}_{2,3} \subset \C^{58}$ has dimension 41 and degree 39.
\end{example}

\begin{proposition}\label{prop:free-group-variety}
For all $d\geq 1$ and $k\geq 1$, the ideal $G_{d,k}$ is prime, so $\mathcal{G}_{d,k}$ is irreducible.
\end{proposition}
\begin{proof}
By \Cref{cor:free} the shuffle algebra $H$ is free, and $G_{d,k}$ is generated by the shuffle relations among coordinates of weight at most $k$. Hence, the quotient $\C[\sigma^{(\leq k)}]/G_{d,k}$ is the truncation of $H$ by weight $k$ or lower, which is a polynomial ring on the generators of $V_{\leq k}$ as a free commutative algebra. A polynomial ring is an integral domain, so $G_{d,k}$ is prime and $\mathcal{G}_{d,k}$ is thus irreducible.
\end{proof}

The \emph{universal ideal} for order $k+1$ tensors of format $d\times\dots\times d \times k!$ is defined as 
\begin{equation}\label{eq:defUniversalVarietyIdeal}U_{d,k}:=G_{d,k}\cap \C[\sigma^{(k)}]
\end{equation}
with the polynomial sub-ring $\C[\sigma^{(k)}]:=\C[\sigma_{(w,\tau)}\mid (w,\tau)\in C,\len(w)=k]$. This is a homogeneous prime ideal in $k!d^k$ variables. 
We define the \emph{$k$-level universal variety for membranes} $\mathcal{U}_{d,k}$ to be the zero set of the ideal $U_{d,k}$ in the affine space $\C^{d\times \dots\times d\times k!}$. 

\begin{example}\label{ex:U23} We revisit \Cref{ex:G23} and compute $U_{2,3}$. 
Here $\C[\sigma^{( 3)}]$ has
$48$ variables. Eliminating
from $G_{2,3}$ the $10$ variables $\sigma_{(w,\tau)}$
with $\len(w)\leq 2$ gives $U_{2,3}$, which 
is minimally generated by $28$ quadrics, the smallest with $90$ terms. The variety $\mathcal{U}_{2,3}$ has dimension 41 and degree 8.
\end{example}

Note that the orbit
sum relations from the previous section do not suffice to cut out
$\mathcal{U}_{d,k}$, as the following example shows.

\begin{example}\label{ex:orbit-sums-do-not-determine-U23}
The orbit sum relations of \Cref{cor:shuffle-relations} are not sufficient to cut out $\mathcal{U}_{d,k}$, already at $d=2$ and $k=3$. Only the partition $(2,1)$ contributes, since $(1,1,1)$ requires three pairwise distinct letters. The tuples $(1,2)$ and $(2,1)$ yield two relations of total degree $9$ in the $48$ coordinates $\sigma_{(w,\tau)}$. Let $I_{\mathrm{orb}}$ be the ideal they generate. With the drop of dimension $\dim\mathcal{V}(I_{\mathrm{orb}})=46>\dim\mathcal{U}_{2,3}=41$,
we see $\mathcal{U}_{2,3}\subsetneq\mathcal{V}(I_{\mathrm{orb}})$ strictly.   
\end{example}

With the shuffle identity (\Cref{thm:shuffle-identity}) we see that the projection of a truncated signature to its highest graded component is finite-to-one. 

\begin{theorem}\label{prop:proojectionResult}
    The projection from $\mathcal{G}_{d,k}$ to $\mathcal{U}_{d,k}$ is generically $k$-to-one. 
\end{theorem}
\begin{proof}
Write $\operatorname{pr}\colon\mathcal{G}_{d,k}\to\mathcal{U}_{d,k}$ for the projection that forgets every coordinate of weight less than $k$. As $U_{d,k}=G_{d,k}\cap\C[\sigma^{(k)}]$ is the corresponding elimination ideal, $\mathcal{U}_{d,k}$ is the Zariski closure of $\operatorname{pr}(\mathcal{G}_{d,k})$ and $\operatorname{pr}$ is dominant. The argument adapts \cite[Theorem 6.1]{bib:AFS2019} to the two-parameter setting, with the reconstruction of the lower path tensors replaced by the descending recursion, derived below. 

For $\xi\in\C^{*}$ let $\delta_\xi$ act on $\C[\sigma^{(\leq k)}]$ by $\delta_\xi\sigma_{(w,\tau)}:=\xi^{\len(w)}\sigma_{(w,\tau)}$. Note that each generator $\sigma_{(w_1,\tau_1)}\sigma_{(w_2,\tau_2)}-\sigma_{(w_1,\tau_1)\shuffle(w_2,\tau_2)}$ of $G_{d,k}$ is homogeneous of weight $\len(w_1w_2)$, so $\delta_\xi$ preserves $\mathcal{G}_{d,k}$ and satisfies $\operatorname{pr}(\delta_\xi\sigma)=\xi^{k}\,\operatorname{pr}(\sigma)$.

Since $\mathcal{G}_{d,k}$ contains signatures of membranes with nonvanishing first graded component, \Cref{thm:shuffle-identity} and \Cref{lem:diagonal-identity} show that the linear form $\Theta(k,\mathbf a)=\sum_{\tau\in\Sigma_k}\sigma_{(\mathbf a^{k},\tau)}$ on $\C[\sigma^{(k)}]$ does not vanish identically on $\mathcal{U}_{d,k}$ for at least one letter $\mathbf a\in\{\w{1},\dots,\w{d}\}$. Fix such a letter, abbreviate $\sigma_{\mathbf a}:=\sigma_{(\mathbf a,1)}$, and set
$
\mathcal{U}_{d,k}^{\circ}:=\{\,s\in\mathcal{U}_{d,k}\mid \Theta(k,\mathbf a)(s)\neq0\,\},
$
a dense open subset of $\mathcal{U}_{d,k}$. We prove that $\operatorname{pr}^{-1}(s)$ has exactly $k$ points for every $s\in\mathcal{U}_{d,k}^{\circ}$.
Let $\sigma\in\mathcal{G}_{d,k}$ with $\operatorname{pr}(\sigma)=s$. By \Cref{thm:shuffle-identity} and \Cref{lem:diagonal-identity},
\begin{equation}\label{eq:fibre-root}
\sigma_{\mathbf a}^{\,k}=\sigma_{(\mathbf a,1)^{\shuffle k}}=k!\,\Theta(k,\mathbf a)(s),
\end{equation}
whose right-hand side is a nonzero scalar determined by $s$. Hence $\sigma_{\mathbf a}$ is one of the $k$ distinct $k$th roots of $k!\,\Theta(k,\mathbf a)(s)$, and in particular $\sigma_{\mathbf a}\neq0$.

The scalar $\sigma_{\mathbf a}$ determines the remaining coordinates of $\sigma$. For $1\leq\ell\leq k-1$ and every $(w,\tau)\in C$ with $\len(w)=\ell$, the shuffle relation of weight $\ell+1$ gives, by \Cref{thm:shuffle-identity}, the descending recursion:
\begin{equation}\label{eq:fibre-recursion}
\sigma_{(w,\tau)}=\sigma_{\mathbf a}^{-1}\,\sigma_{(\mathbf a,1)\shuffle(w,\tau)},
\end{equation}
where $\sigma_{(\mathbf a,1)\shuffle(w,\tau)}$ is a linear form in the coordinates of weight $\ell+1$. Evaluated downwards from $\ell=k-1$ to $\ell=1$, starting from the prescribed top component $\sigma^{(k)}=s$, the recursion \eqref{eq:fibre-recursion} expresses every coordinate of $\sigma$ through $s$ and $\sigma_{\mathbf a}$. Thus $s$ together with the choice of root in \eqref{eq:fibre-root} recovers $\sigma$ uniquely, so $\operatorname{pr}^{-1}(s)$ has at most $k$ points.

In addition, each root is realized. Choose one preimage $\sigma$ of $s$. For a $k$th root of unity $\zeta$ the point $\delta_\zeta\sigma$ lies in $\mathcal{G}_{d,k}$ and satisfies $\operatorname{pr}(\delta_\zeta\sigma)=\zeta^{k}s=s$, with distinguished coordinate $(\delta_\zeta\sigma)_{\mathbf a}=\zeta\,\sigma_{\mathbf a}$. As $\zeta$ runs over the $k$th roots of unity these coordinates take the $k$ distinct values in \eqref{eq:fibre-root}, so the points $\delta_\zeta\sigma$ are pairwise distinct and exhaust $\operatorname{pr}^{-1}(s)$. Therefore $\operatorname{pr}$ is generically $k$-to-one.
\end{proof}

Keeping in mind that we can recover the lower signature levels from a fixed level signature tensor, we focus on the $k$th level.  
We study the relations from $U_{d,k}$ and additional relations, arising by a restriction to piecewise bilinear and polynomial membranes. For this we consider \eqref{eq:signature_tensor} as a polynomial map from a specific family of membranes to tensors, and verify that its image is a semialgebraic set. 
Central to this construction is the following action of matrices on membranes and its corresponding notion of dictionaries. This approach is inspired by the path setting \cite{bib:PSS2019}. 

\begin{definition}
Let $A \in \mathbb{R}^{d \times m}$ and 
$X : [0,1]^2 \to \mathbb{R}^m$ be a membrane.
We define the action of $A$ on $X$ as the membrane 
$A * X : [0,1]^2 \to \mathbb{R}^d$, 
pointwise for all $(s,t)\in[0,1]^2$, by
$(A * X)(s,t) := A\,X(s,t)$, 
i.e., in coordinates, we have for every $1\leq i\leq d$, 
\begin{equation*}
(A * X)_i(s,t)
=
\sum_{j=1}^m A_{ij} X_j(s,t). 
\end{equation*}
\end{definition}

A \emph{polynomial membrane of order $(m,n)$} is a membrane whose component functions are polynomials of multidegree bounded by $(m,n)$. 
According to \cite[Equation (2)]{LS24}, this family of membranes admits a dictionary, the \emph{moment membrane} 
    \begin{align}\label{eq:into_momentMembrane}
        \mathsf{Mom}^{m,n}: [0,1]^2 &\to \mathbb R^{mn},(s,t)\mapsto(s,s^2,\dots,s^m)\otimes(t,t^2,\dots,t^n).
    \end{align}
Using the language of dictionaries, the family of polynomial membranes of order $(m,n)$ with trivial linear part is given by $\R^{d\times mn}*\mathsf{Mom}^{m,n}$. 

\begin{example}\label{ex:polynmialMembrane} For a matrix $A$ of size $2\times 4$, we obtain 
  \begin{equation*}A*\mathsf{Mom}^{2,2}:[0,1]^2\to\R^2,\;(s,t)\mapsto\begin{pmatrix}
    A_{11}st+A_{12}st^2+A_{13}s^2t+A_{14}s^2t^2\\
    A_{21}st+A_{22}st^2+A_{23}s^2t+A_{24}s^2t^2
\end{pmatrix}
\end{equation*}
that is, a polynomial membrane of order $(2,2)$ with trivial linear part in $s$ and $t$. For a visualization of such a membrane, with explicit $A$, we refer to \cite[Figure 3]{LS24}. 
\end{example}

We now read the $k$th graded component of the signature \eqref{eq:signature_tensor} as a map and consider its image
\begin{equation}\label{eq:defP_im_dkmn}
\mathcal{P}_{d,k,(m,n)}^{\mathsf{im}}
:=
\left\{\sigma^{(k)}(X)\; \begin{array}{|l}
X:[0,1]^2\to\mathbb{R}^d \\\text{polynomial }\\\text{of order }(m,n)
\end{array}
\right\}
\end{equation}
when restricted to polynomial membranes of a fixed order. 
To study the polynomials that vanish on these sets, we move
to the algebraically closed field $\C$ and denote by $\mathcal{P}_{d,k,(m,n)}$ the Zariski closure of \eqref{eq:defP_im_dkmn}.   Clearly, we have 
$\mathcal{P}_{d,k,(m,n)}\subseteq\mathcal{U}_{d,k}$. 

We define an action on the $k$th graded component of signatures, so that we can formulate a corresponding equivariance result, and thus represent \eqref{eq:defP_im_dkmn} as a semialgebraic set. 

\begin{definition}\label{def:congruence}
Let $A \in \mathbb{R}^{d \times m}$ and
$G \in \mathbb{R}^{m \times \dots \times m \times k!}$ be any $(k+1)$-tensor.
We define the \emph{action} of $A$ on $G$ as the $(k+1)$-tensor
$A * G \in \mathbb{R}^{d \times \dots \times d \times k!}$ given by
\begin{equation*}
(A * G)_{i_1\dots i_kt}
:=
\sum_{j_1,\dots,j_k=1}^{m}
A_{i_1j_1} \dots A_{i_kj_k}\,
G_{j_1\dots j_kt}.
\end{equation*}
\end{definition}

\begin{lemma}\label{lem:equivariance}
   For all $X : [0,1]^2 \to \mathbb{R}^m$, 
$A \in \mathbb{R}^{d\times m}$, and $k\geq 1$, 
\begin{equation*}
\sigma^{(k)}(A*X)
=
A * \sigma^{(k)}(X).
\end{equation*}
\end{lemma}

\begin{proof}
We generalize \cite[Lemma 2.6]{LS24} for integration over a permuted simplex.
By \Cref{def:signature_two_parameter}, the components of the signature at slice $\tau$ of order $k$ are
\begin{equation*}
\sigma^{(\tau)}_{j_1\dots j_k}(X)
=
\int\limits_{\substack{0\leq s_1\leq\dots\leq s_k\leq 1\\
0\leq t_1\leq\dots\leq t_k\leq 1}}
\partial_{12}X_{j_1}(s_1,t_{\tau(1)})
\dots
\partial_{12}X_{j_k}(s_k,t_{\tau(k)})
\;\mathrm{d}s_1\mathrm{d}t_1\dots\mathrm{d}s_k\mathrm{d}t_k.
\end{equation*}
Using the congruence action,
$\partial_{12}(A*X)_i(s,t)
=
\sum_{j=1}^m A_{ij}\partial_{12}X_j(s,t)$, so that
\begin{align*}
\sigma^{(\tau)}_{i_1\dots i_k}(A*X)
&=
\int\limits_{\substack{0\leq s_1\leq\dots\leq s_k\leq 1\\
0\leq t_1\leq\dots\leq t_k\leq 1}}
\prod_{r=1}^k
\left(
\sum_{j_r=1}^m
A_{i_rj_r}\,
\partial_{12}X_{j_r}(s_r,t_{\tau(r)})
\right)
\mathrm{d}s_1\mathrm{d}t_1\dots\mathrm{d}s_k\mathrm{d}t_k
\\
&=
\sum_{j_1,\dots,j_k=1}^m
A_{i_1j_1}\dots A_{i_kj_k}
\,
\sigma^{(\tau)}_{j_1\dots j_k}(X)
=(A * \sigma^{(\tau)}(X))_{i_1\dots i_k}
\end{align*}
where the second equality expands the product over $r$ and applies Fubini's theorem to move the finite sum outside the integral.
\end{proof}

\begin{corollary}\label{cor:realizationPimdkmn}
    $\mathcal{P}_{d,k,(m,n)}^{\mathsf{im}}=\mathbb{R}^{d\times mn}*\sigma^{(k)}(\mathsf{Mom}^{m,n})$.
\end{corollary}
\begin{proof}
    This follows with \Cref{lem:equivariance} and \eqref{eq:defP_im_dkmn}. 
\end{proof}

For actual computations, it remains to derive a closed-form expression for the $k$th graded component of the signature for moment membranes, $\sigma^{(k)}(\mathsf{Mom}^{m,n})$. 
For this, we recall the notion of a product membrane from \cite[Section 3]{LS24}. 
We introduce a second action on tensors: 
for $\tau\in\Sigma_k$ and a $k$-tensor $T\in(\R^{d})^{\otimes k}$ we write
$\tau\cdot T$ for the $k$-tensor with entries
\begin{equation}\label{not:tensorTransposition}(\tau\cdot T)_{i_1\dots i_k}:=T_{i_{\tau^{-1}(1)}\dots i_{\tau^{-1}(k)}}.
\end{equation}

For a path $X:[0,1]\to \R^d$ let $\sigma^{(k)}(X)$ denote its $k$-level path signature tensor of $X$, e.g., compare \cite[Equation (6)]{Friz_Victoir_2010} or \cite[Equation (1)]{bib:AFS2019}. 
 
\begin{proposition}\label{twoparamProdMemb}
Let $V,W$ be finite-dimensional vector spaces and let
$X:[0,1]\to V$ and $Y:[0,1]\to W$ be smooth paths. Let
$X\boxtimes Y:[0,1]^2\to V\otimes W$ denote their \emph{product membrane},
$(X \boxtimes Y)(s,t) := X(s) \otimes Y(t)$ with $(s,t) \in [0,1]^2$.
Then, for every $\tau\in\Sigma_k$,
\begin{equation}\label{eq:prodMembraneFactorization}
\sigma^{(\tau)}(X \boxtimes Y) = \sigma^{(k)}(X) \otimes \bigl(\tau\cdot\sigma^{(k)}(Y)\bigr).
\end{equation}
Explicitly, indexing
$V\otimes W$ by pairs,
\begin{equation*}
\sigma^{(\tau)}_{(i_1',i_1'')\dots(i_k',i_k'')}(X\boxtimes Y)
=\sigma^{(k)}(X)_{i_1'\dots i_k'}\cdot\sigma^{(k)}(Y)_{i''_{\tau^{-1}(1)}\dots i''_{\tau^{-1}(k)}} .
\end{equation*}
\end{proposition}
 
\begin{proof}
In coordinates, $(X\boxtimes Y)_{(i',i'')}(s,t)=X_{i'}(s)\,Y_{i''}(t)$, so the
mixed partial derivative separates the two variables,
\begin{equation}\label{eq:boxtimesDerivative}
\partial_{12}(X\boxtimes Y)_{(i',i'')}(s,t)=\partial X_{i'}(s)\,\partial Y_{i''}(t).
\end{equation}
Inserting \eqref{eq:boxtimesDerivative} into \Cref{def:2-param-signature} and
using that the domain is a product of two simplices, the integral factorizes as
$\sigma^{(\tau)}_{(i_1',i_1'')\dots(i_k',i_k'')}(X\boxtimes Y)=\sigma_k(X)_{i_1'\dots i_k'}\cdot Q$ with
\begin{align*}
Q&:=\int\limits_{0\le t_1\le\dots\le t_k\le 1}
\partial Y_{i_1''}(t_{\tau(1)})\dots\partial Y_{i_k''}(t_{\tau(k)})\;\mathrm{d}t_1\dots\mathrm{d}t_k .
\end{align*}
The first factor is the $k$th level of the path signature of $X$.  
For the second,
reorder the factors according to the index of the integration variable: putting
$j=\tau(r)$, that is, $r=\tau^{-1}(j)$, the integrand becomes
$\prod_{j=1}^{k}\partial Y_{i''_{\tau^{-1}(j)}}(t_j)$, hence
\begin{equation*}
Q=\sigma_k(Y)_{i''_{\tau^{-1}(1)}\dots i''_{\tau^{-1}(k)}}
 =\bigl(\tau\cdot\sigma_k(Y)\bigr)_{i_1'' \dots i_k''}
\end{equation*}
and collecting the two factors gives the asserted identity.
\end{proof}

\Cref{twoparamProdMemb} reduces the signature of a product membrane
to the path signatures of its two factors. In particular, whenever both
factors are paths whose signatures are known in closed form, so is the signature
of the resulting membrane. The dictionary for polynomial membranes factorizes in this sense. 

\begin{example}\label{ex:productMembrane}
The moment membrane $\mathsf{Mom}^{m,n}=\mathsf{Mom}^{m}\boxtimes\mathsf{Mom}^{n}$ is a product membrane according to \cite[Section 3]{LS24}, where 
$\mathsf{Mom}^{m}:[0,1]\to\R^{m},\;s\mapsto(s,s^{2},\dots,s^{m})$
 denotes the $m$th moment curve. 
\end{example}

We illustrate \Cref{twoparamProdMemb} at truncation level $k=2$ via Kronecker matrix products. 

\begin{corollary}\label{cor:productMembranes}
The $2$-level signature of the moment membrane is 
$$\sigma^{(2)}(\mathsf{Mom}^{m,n})=\left( S_m \otimes S_n\right)\oplus \left(S_m \otimes S_n^\top\right)$$
where $S_d:=(\frac{q}{p+q})_{p,q}$ is the $d\times d$ signature matrix of the $d$th moment curve. 
\end{corollary}
\begin{proof} 
Using the explicit formula $\sigma^{(2)}(\mathsf{Mom}^m)=S_m$ from \cite[Example 2.2]{bib:AFS2019}, we obtain the claim with \Cref{twoparamProdMemb} and the factorization in \Cref{ex:productMembrane}, i.e., we apply \eqref{eq:prodMembraneFactorization} with $k=2$. For $\tau=\mathsf{id}_2$ the
tensor transposition is the identity, and for $\tau=(2\,1)$ it is the matrix
transpose by \eqref{not:tensorTransposition}. 
\end{proof}
 
\begin{example}\label{ex:momentMembrane}
We revisit the polynomial membrane $A*\mathsf{Mom}^{2,2}$ from \Cref{ex:polynmialMembrane}, and compute its signature coefficient at $(\w{11},2\,1)$. 
For this, we consider the dictionary  
$\sigma^{(2)}(\mathsf{Mom}^{2,2})=
\left(\frac{1}{36}\begin{psmallmatrix} 3 & 2\\ 4 & 3\end{psmallmatrix}\otimes \begin{psmallmatrix} 3 & 2\\ 4 & 3\end{psmallmatrix}
\right)\oplus\left(
\frac{1}{36}\begin{psmallmatrix} 3 & 4\\ 2 & 3\end{psmallmatrix}\otimes \begin{psmallmatrix} 3 & 2\\ 4 & 3\end{psmallmatrix}\right)$
from \Cref{cor:productMembranes}, and obtain via congruence from \Cref{lem:equivariance} that
\begin{align*}
\sigma^{(2\,1)}(A*\mathsf{Mom}^{2,2})_{11}
=\frac14A_{1 1}^2 &+ 
\frac12A_{1 1}A_{1 2} + 
\frac12A_{1 1}A_{1 3} + 
\frac49A_{1 1}A_{1 4} + 
\frac14A_{1 2}^2 \\
&+\frac59A_{1 2}A_{1 3} + 
\frac12A_{1 2}A_{1 4} + 
\frac14A_{1 3}^2 + 
\frac12A_{1 3}A_{1 4} + 
\frac14A_{1 4}^2,
\end{align*}
which is a polynomial in the entries of $A$. The remaining three entries of $\sigma^{(2\,1)}(A*\mathsf{Mom}^{2,2})$ are similar. For the four entries of the 2-level id-signature  matrix 
$\sigma^{(1\,2)}(A*\mathsf{Mom}^{2,2})$ we refer to \cite[Example 3.7]{LS24}. 
\end{example}

In our implementation \cite{RS26}, we provide the resulting formulas (according to \Cref{twoparamProdMemb}) for higher signature tensors of dictionaries, $\sigma^{(k)}(\mathsf{Mom}^{m,n})$. In particular, with \Cref{lem:equivariance} we can compute any signature coefficient  of a polynomial membrane via equivariance. 
We now study the associated varieties.  

\begin{proposition}\label{prop:Pdkmn_irred}
The varieties $\mathcal{P}_{d,k,(m,n)}\subseteq\mathcal{U}_{d,k}$ are irreducible and can be arranged in the following grid of inclusions,  
\begin{equation}\label{eq:grid_of_varieties}
    \begin{matrix}
        \mathcal{P}_{d,k,(1,1)} & \subseteq & \mathcal{P}_{d,k,(1,2)} & \subseteq & \mathcal{P}_{d,k,(1,3)}& \subseteq & \dots & \subseteq &\mathcal{P}_{d,k,(1,N)}\\
        \rotatebox[origin=c]{270}{$\subseteq$} 
        &&\rotatebox[origin=c]{270}{$\subseteq$}
        &&\rotatebox[origin=c]{270}{$\subseteq$}
        &&&&
        \rotatebox[origin=c]{270}{$\subseteq$}\\
        \mathcal{P}_{d,k,(2,1)} & \subseteq & \mathcal{P}_{d,k,(2,2)} & \subseteq & \mathcal{P}_{d,k,(2,3)}& \subseteq & \dots & \subseteq &\mathcal{P}_{d,k,(2,N)}\\
        \rotatebox[origin=c]{270}{$\subseteq$} 
        &&\rotatebox[origin=c]{270}{$\subseteq$}
        &&\rotatebox[origin=c]{270}{$\subseteq$}
        &&&&
        \rotatebox[origin=c]{270}{$\subseteq$}\\
        \vdots &  & \vdots &  & \vdots&  & \ddots & &\vdots\\
        \rotatebox[origin=c]{270}{$\subseteq$} 
        &&\rotatebox[origin=c]{270}{$\subseteq$}
        &&\rotatebox[origin=c]{270}{$\subseteq$}
        &&&&
        \rotatebox[origin=c]{270}{$\subseteq$}\\
         \mathcal{P}_{d,k,(M,1)} & \subseteq & \mathcal{P}_{d,k,(M,2)} & \subseteq & \mathcal{P}_{d,k,(M,3)}& \subseteq & \dots & \subseteq &\mathcal{P}_{d,k,(M,N)}
    \end{matrix}
    \end{equation}
    where there exist $M,N\in\mathbb{N}$ such that 
    \begin{align*}
\mathcal{P}_{d,k,(m,1)}\subseteq\dots\subseteq\mathcal{P}_{d,k,(m,N-1)}\subseteq\mathcal{P}_{d,k,(m,N)}=\mathcal{P}_{d,k,(m,N+1)}&=\dots\quad
        \\
        \mathcal{P}_{d,k,(1,n)}\subseteq\dots\subseteq\mathcal{P}_{d,k,(P-1,n)}\subseteq\mathcal{P}_{d,k,(P,n)}=\mathcal{P}_{d,k,(P+1,n)}&=\dots
    \end{align*}
for all  $m,n\in\mathbb{N}$. 
    \end{proposition}
\begin{proof}
  Analogously as in \cite[Theorem 5.6]{bib:AFS2019}, the  $\mathcal{P}^{\mathsf{im}}_{d,k,(m,n)}$ are  images of a polynomial map, see \Cref{cor:realizationPimdkmn}. 
  Therefore, with the Tarski-Seidenberg theorem (e.g., see \cite[Theorem 4.17]{michalek2021invitation}), they are semialgebraic and their closures $\mathcal{P}_{d,k,(m,n)}$ are irreducible varieties. 
  Any membrane of order $(m,n)$ has also order $(m+1,n)$ or $(m,n+1)$, and taking closure preserves inclusions. 
Among these inclusions, only finitely many can be
strict, since each set is an irreducible variety and the ambient dimension is bounded by $k!d^k$. Hence, there exist
positive integers $M$ and $N$ such that the grid \eqref{eq:grid_of_varieties} stabilizes.
\end{proof}

In \Cref{sec:ChenChow} we focus on the second graded component and show that the grid \eqref{eq:grid_of_varieties} stabilizes in the $2$-level universal variety, i.e., $\mathcal{P}_{d,2,(M,N)}=\mathcal{U}_{d,2}$. We show that its dimension is $d(3d+1)/2$, that is, strictly smaller than the dimension of the ambient space, $2d^2$.

In \Cref{sec:Witt} we determine the dimension of the $k$-level universal variety $\mathcal{U}_{d,k}$ and thus strictly bound the dimension of $\mathcal{P}_{d,k,(M,N)}$ below the ambient dimension $k!d^k$. Furthermore, we provide computational evidence with \cite{RS26} that the grid \eqref{eq:grid_of_varieties} stabilizes in $\mathcal{U}_{d,k}$. 

We close this section by recalling a second important class of membranes:
a \emph{piecewise bilinear membrane of order $(m,n)$} is a membrane $X$ such that there are some $m,n \in \N$ with the property that $X$ is biaffine on all rectangles $[\frac{i-1}{m},\frac{i}{m}] \times [\frac{j-1}{n},\frac{j}{n}]$. Again, this family of membranes admits a dictionary, and thus the semialgebraic set 
\begin{equation*}
\mathcal{L}_{d,k,(m,n)}^{\mathsf{im}}
:=
\left\{\sigma^{(k)}(X)\; \begin{array}{|l}
X:[0,1]^2\to\mathbb{R}^d \\\text{piecewise bilinear}\\\text{of order }(m,n)
\end{array}
\right\}
\end{equation*}
with Zariski closure $\mathcal{L}_{d,k,(m,n)}$. In particular, the varieties $\mathcal{L}_{d,k,(m,n)}$ are irreducible and can be arranged via an analogous grid as \eqref{eq:grid_of_varieties}.
We show that its second graded component coincides with the variety for polynomial membranes. 

\begin{theorem}\label{thm:polyVsPwLin}
For all $d,m,n\in\mathbb{N}$,
\begin{equation*}
\mathcal{L}_{d,2,(m,n)} = \mathcal{P}_{d,2,(m,n)}.
\end{equation*}
\end{theorem}

\begin{proof}
We show $\mathcal{L}^{\mathsf{im}}_{d,2,(m,n)} = \mathcal{P}^{\mathsf{im}}_{d,2,(m,n)}$, i.e., already the semialgebraic sets agree. The class of piecewise bilinear membranes admits a product membrane $\mathsf{Axis}^{m,n}=\mathsf{Axis}^{m}\boxtimes\mathsf{Axis}^{n}$ as its dictionary, see \cite[Definition 3.14 and Lemma 3.17]{LS24}, where $\mathsf{Axis}^{m}:[0,1]\to\R^m$ is the $m$th canonical axis path according to \cite[Example 2.1]{bib:AFS2019}. 
By definition of the two sets, it suffices to find an invertible $H\in\mathbb{R}^{mn\times mn}$ such that $H*\sigma^{(2)}(\mathsf{Mom}^{m,n}) = \sigma^{(2)}(\mathsf{Axis}^{m,n})$. By \Cref{twoparamProdMemb},
the signature matrix of a product membrane factors into path signatures of the corresponding paths, 
so with $H:=H_{m}\otimes H_{n}$ for transformations $H_n$ according to \cite[Theorem 3.3]{bib:AFS2019} from moment curves to canonical axis paths, we can use the mixed-product property (e.g., \cite[Lemma 4.2.10]{HornJohnson1991Topics}) on both slices above.
\end{proof}

For higher truncation level, similarly as  in the path case (e.g., see \cite[Theorem 5.7]{bib:AFS2019}), the varieties $\mathcal{L}_{d,k,(m,n)}$ and $\mathcal{P}_{d,k,(m,n)}$ are distinct in general. 

\begin{example}\label{ex:polyVsPwLin-k3}
Take $d=2$, $k=3$ and $(m,n)=(1,2)$, so both $\mathcal{P}_{2,3,(1,2)}$ and
$\mathcal{L}_{2,3,(1,2)}$ live in the ambient space $\C^{48}$. Let
$P_{2,3,(1,2)}$ and $ L_{2,3,(1,2)}$ denote their
vanishing ideals, respectively.
Both ideals cut out irreducible affine varieties of dimension $4$ and degree $6$, and each is minimally generated by $40$ linear and $9$ quadratic polynomials. The variety $\mathcal{P}_{2,3,(1,2)}\cup\mathcal{L}_{2,3,(1,2)}$ has degree $12$. 
Thus
\begin{equation*}
\mathcal{P}_{2,3,(1,2)}\ \neq\ \mathcal{L}_{2,3,(1,2)},
\end{equation*}
with neither variety contained in the other. In fact, the quadric
\begin{align*}
&3\sigma_{(\w{211},2\,3\,1)}\sigma_{(\w{222},3\,2\,1)}
-3\sigma_{(\w{211},3\,2\,1)}\sigma_{(\w{222},3\,2\,1)}
-3\sigma_{(\w{221},2\,1\,3)}\sigma_{(\w{221},3\,2\,1)}\\
&\qquad+\sigma_{(\w{221},3\,1\,2)}^{2}
+\sigma_{(\w{221},3\,1\,2)}\sigma_{(\w{221},3\,2\,1)}
+\sigma_{(\w{221},3\,2\,1)}^{2}
\end{align*}
lies in $\mathcal{L}_{2,3,(1,2)}$ but not in $\mathcal{P}_{2,3,(1,2)}$. The full list of generators for each ideal is included in the computation available in \cite{RS26}.
\end{example}

\section{An algebraic Chen-Chow theorem for level \texorpdfstring{$k=2$}{k=2}}\label{sec:ChenChow}

The second graded component of a membrane signature is a pair of $d\times d$ matrices, whose congruence normal forms are classically available \cite[Theorem 1.1]{HORN20061010}. This has successfully been used in \cite[Theorem 3.4]{bib:AFS2019} to describe the matrix signature varieties for paths. For membranes, the associated variety of the id-signature matrix is already known to attain full dimension $d^2$, see \cite[Theorem 4.13]{LS24}. In this section, we extend this result to the full second graded component, i.e., a pair of matrices. Furthermore, we determine the dimension of the $2$-level universal variety via a parametrized family of Gr\"obner bases.

In \Cref{ex:shuffle-k2} we provided an explicit relation based on symmetrized orbit sums in the $2$nd graded component of any signature, see also \Cref{cor:shuffle-relations} for the general result. Using the identification from \eqref{eq:signature_tensor}, we formalize this in the polynomial ring $\C[\sigma^{(3)}]$ over $2d^2$ variables via 
\begin{equation*}
r_{i,j}:=(\sigma_{(\w{ij},1\,2)}+\sigma_{(\w{ij},2\,1)}+\sigma_{(\w{ji},1\,2)}+\sigma_{(\w{ji},2\,1)})^2-4(\sigma_{(\w{ii},1\,2)}+\sigma_{(\w{ii},2\,1)})(\sigma_{(\w{jj},1\,2)}+\sigma_{(\w{jj},2\,1)})
\end{equation*}
that is, $r_{i,j}\in U_{d,k}$ 
is contained in the vanishing ideal \eqref{eq:defUniversalVarietyIdeal} of the universal variety $\mathcal{U}_{d,k}$. 

\begin{proposition}\label{lem:gb_k2_univVar}
For $d\geq 2$, the set $\{r_{i,j}\mid i<j\}$ forms a Gr\"obner basis of the ideal  $\langle r_{i,j}\mid 1\leq i,j\leq d\rangle$ with respect to any degree lexicographic order.
\end{proposition}

\begin{proof}
Let $M$ be the symmetric $d\times d$ matrix of linear forms
\begin{equation*}
M_{i,j}=\begin{cases}
2(\sigma_{(\w{ii},1\,2)}+\sigma_{(\w{ii},2\,1)}) & \text{if } i=j, \\
\sigma_{(\w{ij},1\,2)}+\sigma_{(\w{ij},2\,1)}+\sigma_{(\w{ji},1\,2)}+\sigma_{(\w{ji},2\,1)} & \text{if } i\neq j.
\end{cases}
\end{equation*}
Then, $r_{i,j}=M_{i,j}^2-M_{i,i}M_{j,j}$ is the negative of the $2\times2$ minor of $M$ indexed by $(i,j)$. Hence $\{r_{i,j}\mid i<j\}$ consists of the negatives of the $2\times2$ minors of a symmetric matrix, which form a Gr\"obner basis with respect to any degree lexicographic order \cite[Theorem 1]{sturmfels1990grobner}.
\end{proof}

\begin{corollary}\label{prop:GBforshufflerelationsk2}
For all $d\geq 2$, we have $\dim\mathcal{V}(r_{i,j}\mid 1\leq i,j\leq d)=\frac12d(3d+1)$. 
\end{corollary}
\begin{proof}
By \Cref{lem:gb_k2_univVar}, we obtain a Gr\"obner basis with $d(d-1)/2$ generators. Fixing the order so that the leading monomial of $r_{i,j}$ is $\sigma_{(\w{ij},1\,2)}^2$, the initial ideal is $\langle \sigma_{(\w{ij},1\,2)}^2\mid i<j\rangle$, with radical $\langle \sigma_{(\w{ij},1\,2)}\mid i<j\rangle$. The dimension of a variety equals that of its radical ideal quotient, which equals that of the initial ideal quotient, that is $2d^2-d(d-1)/2$.
\end{proof}

\begin{corollary}\label{cor:degree-shuffle-k2}
For all $d\geq 2$, $\mathcal{V}(r_{i,j}\mid 1\leq i,j\leq d)\subset\C^{2d^2}$ has degree $2^{d(d-1)/2}$.
\end{corollary}

\begin{proof}
With \Cref{lem:gb_k2_univVar} it suffices to compute the degree of the initial ideal generated by $d(d-1)/2$ monomials $\sigma_{(\w{ij},1\,2)}^2$ of degree $2$ with pairwise distinct variables $i<j$, so the claim follows with Bézout's theorem, e.g., compare \cite[Theorem 2.16]{michalek2021invitation}.  
\end{proof}

Clearly, the universal variety $\mathcal{U}_{d,k}$ is a subvariety of the variety of symmetrized orbit sums from \Cref{prop:GBforshufflerelationsk2}. Furthermore, the varieties $\mathcal{P}_{d,2,m}$ of polynomial membranes are subvarieties of $\mathcal{U}_{d,k}$. Our main result is that $\mathcal{P}_{d,2,(m,n)}$ eventually fills the variety of symmetrized orbit sums, i.e., the grid \eqref{eq:grid_of_varieties} stabilizes in the $2$-level universal variety.  

\begin{theorem}\label{thm:dimensionUniversalVariety}
The $2$-level universal variety satisfies
\begin{equation*}
\mathcal{U}_{d,2}=\bigcup_{m,n\to\infty}\mathcal{P}_{d,2,(m,n)}=\mathcal{V}\bigl(r_{i,j}\mid 1\leq i,j\leq d\bigr),
\end{equation*}
has degree $2^{d(d-1)/2}$, and its dimension is the second pentagonal number,
\begin{equation*}
\dim(\mathcal{U}_{d,2})=\frac{1}{2}d(3d+1).
\end{equation*}
\end{theorem}

The proof is carried out in the remainder of this section. 
In contrast to \cite[Theorem 5.6]{bib:AFS2019} we cannot use the Chen-Chow theorem for path signatures, e.g., see \cite[Theorem 7.28]{Friz_Victoir_2010}, since this is not yet established for membranes. Note that the following alternative approach leads to an algebraic Chen-Chow result for membranes  in the second graded component, \Cref{cor:algChenChow-k2}. 
The novel approach computes the rank of the Jacobian of the signature parametrization at the dictionary $\mathsf{Mom}^{d,d}$ from \eqref{eq:into_momentMembrane}. We first bring its two slices $\sigma^{(1,2)}(\mathsf{Mom}^{d,d})$ and $\sigma^{(2,1)}(\mathsf{Mom}^{d,d})$ to a direct sum of the elementary blocks, 
\begin{equation*}
H_2(1):=\begin{pmatrix}0&1\\1&0\end{pmatrix} \quad \text{ and }\quad H_2(-1):=\begin{pmatrix}0&-1\\1&0\end{pmatrix}
\end{equation*}
which are used in the construction of congruence normal forms, 
\cite[Theorem 1]{HORN20061010}. 
We  treat the two parities of $d$ separately.

\begin{lemma}\label{lemma:sigma12odd}
For $d$ odd, set $\kappa:=\frac{d-1}{2}$. With
\begin{align*}
N_1 &:= 1 \oplus H_2(-1)^{\oplus 2\kappa} \oplus H_2(1)^{\oplus \kappa^2} \oplus (-H_2(1))^{\oplus \kappa^2} ,
\\
N_2 &:= 1 \oplus H_2(-1)^{\oplus \kappa} \oplus (-H_2(-1))^{\oplus \kappa} \oplus (-H_2(1))^{\oplus \kappa^2} \oplus H_2(1)^{\oplus \kappa^2} ,
\end{align*}
we have the simultaneous congruence, 
\begin{equation*}
    N_1\sim \sigma^{(1\,2)}(\mathsf{Mom}^{d,d}) \quad\text{ and }\quad N_2 \sim \sigma^{(2\,1)}(\mathsf{Mom}^{d,d}).
\end{equation*}

\end{lemma}
\begin{proof} 
By \cite[Theorem 3.4]{bib:AFS2019}, we know 
$$\sigma^{(2)}(\mathsf{Mom}^d)\sim 1\oplus H_2(-1)^{\oplus \kappa}=:Y.$$
By the decomposition in \Cref{twoparamProdMemb}, $\sigma^{(1\,2)}(\mathsf{Mom}^{d,d})=\sigma^{(2)}(\mathsf{Mom}^d)\otimes\sigma^{(2)}(\mathsf{Mom}^d)$ and $\sigma^{(2\,1)}(\mathsf{Mom}^{d,d})=\sigma^{(2)}(\mathsf{Mom}^d)\otimes\sigma^{(2)}(\mathsf{Mom}^d)^\top$. A congruence $C=PMP^\top$ shows that $C\otimes C=Q(M\otimes M)Q^\top$ and $C\otimes C^\top=Q(M\otimes M^\top)Q^\top$ with $Q=P\otimes P$, by the mixed-product property \cite[Lemma 4.2.10]{HornJohnson1991Topics}, so $\sigma^{(1\,2)}(\mathsf{Mom}^{d,d})\sim Y\otimes Y$ and $\sigma^{(2\,1)}(\mathsf{Mom}^{d,d})\sim Y\otimes Y^\top$.
Expanding 
$$Y\otimes Y=1\oplus H_2(-1)^{\oplus 2\kappa}\oplus\bigl(H_2(-1)^{\oplus \kappa}\otimes H_2(-1)^{\oplus \kappa}\bigr)$$
the last term splits into $\kappa^2$ blocks $H_2(-1)\otimes H_2(-1)$, each congruent to $H_2(1)\oplus(-H_2(1))$ by a row/column permutation, using $H_2(-1)^\top=-H_2(-1)$, which gives $N_1$. Clearly, we have  $Y^\top=1\oplus(-H_2(-1))^{\oplus \kappa}$, so the expansion of $Y\otimes Y^\top$ gives $N_2$, with each block  $H_2(-1)\otimes(-H_2(-1))\sim(-H_2(1))\oplus H_2(1)$.
\end{proof}

The even case is analogous, carrying another block
\begin{equation*}
\tilde\Gamma_2:=\begin{pmatrix}1&-1\\1&0\end{pmatrix}
\end{equation*}
that realizes the representative in \Cref{lemma:sigma12even}. We use $\tilde\Gamma_2$ instead of the original congruent block $\Gamma_2$ from \cite[Theorem 1]{HORN20061010} for technical reasons in the proof of \Cref{thm:dimensionUniversalVariety}, see also \Cref{rem:whyGammaTilde} for further details. 

\begin{lemma}\label{lemma:sigma12even}
For $d$ even, set $\kappa:=\frac{d-2}{2}$. For
\begin{align*}
N_1 &:= (\tilde\Gamma_2\otimes\tilde\Gamma_2) \oplus (\tilde\Gamma_2\otimes H_2(-1))^{\oplus \kappa} \oplus (H_2(-1)\otimes\tilde\Gamma_2)^{\oplus \kappa} \oplus H_2(1)^{\oplus \kappa^2} \oplus (-H_2(1))^{\oplus \kappa^2},
\\
N_2 &:= (\tilde\Gamma_2\otimes\tilde\Gamma_2^\top) \oplus (\tilde\Gamma_2\otimes(-H_2(-1)))^{\oplus \kappa} \oplus (H_2(-1)\otimes\tilde\Gamma_2^\top)^{\oplus \kappa} \oplus (-H_2(1))^{\oplus \kappa^2} \oplus H_2(1)^{\oplus \kappa^2},
\end{align*}
we have the simultaneous congruence, 
\begin{equation*}
N_1 \sim \sigma^{(1\,2)}(\mathsf{Mom}^{d,d}) \quad\text{ and }\quad N_2  \sim \sigma^{(2\,1)}(\mathsf{Mom}^{d,d}).
\end{equation*}
\end{lemma}
\begin{proof}
Here $\sigma^{(2)}(\mathsf{Mom}^d)\sim\Gamma_2\oplus H_2(-1)^{\oplus \kappa}$ with  \cite[Theorem 3.4]{bib:AFS2019}. Since $\tilde\Gamma_2\sim\Gamma_2$, applying this congruence block-diagonally gives $\sigma^{(1\,2)}(\mathsf{Mom}^d)\sim Y:=\tilde\Gamma_2\oplus H_2(-1)^{\oplus \kappa}$. Passing to $Y\otimes Y$ and $Y\otimes Y^\top$ via \Cref{twoparamProdMemb} and expanding block by block is identical to \Cref{lemma:sigma12odd}, with $\tilde\Gamma_2$ in place of the leading $1$.
\end{proof}

\begin{proof}[Proof of \Cref{thm:dimensionUniversalVariety}]
Set $p:=d^2$. By equivariance (\Cref{lem:equivariance}) and \Cref{cor:productMembranes}, the second universal variety is the Zariski closure of the image of the congruence orbit
\begin{equation*}
\phi:\ \R^{d\times p}\longrightarrow\R^{p\times p}\times\R^{p\times p},
\qquad
\phi(A)=\bigl(AN_1A^\top,\ AN_2A^\top\bigr),
\end{equation*}
with $N_1:=\sigma^{(1\,2)}(\mathsf{Mom}^{d,d})$ and $N_2:=\sigma^{(2\,1)}(\mathsf{Mom}^{d,d})$. Thus $\dim\mathcal{U}_{d,2}$ equals the generic rank of the Jacobian $J(A)$ of $\phi$. This rank is invariant under a simultaneous congruence $N_i\mapsto QN_iQ^\top$, so we take $N_1,N_2$ in the normal form of \Cref{lemma:sigma12odd,lemma:sigma12even}. Since $\operatorname{im}\phi\subseteq\mathcal{V}(G_d)$, \Cref{prop:GBforshufflerelationsk2} yields
\begin{equation}\label{eq:upperbound}
\operatorname{rank}J(A)\ \le\ \dim\mathcal{V}(G_d)\ =\ \tfrac{1}{2}d(3d+1).
\end{equation}
As the generic rank is at least the rank at any given point, it suffices to exhibit a single $A$ at which \eqref{eq:upperbound} is an equality. We treat $d$ odd; the case $d$ even is analogous, with the leading scalar block replaced by $\tilde\Gamma_2$ as in \Cref{lemma:sigma12even}, see also \Cref{rem:whyGammaTilde}. 

Differentiating at $A\in\R^{d\times p}$, for $i\in\{1,2\}$,
\begin{equation*}
\mathrm{d}(AN_iA^\top)=( \mathrm{d}A)\,N_iA^\top+AN_i\,(\mathrm{d} A)^\top.
\end{equation*}
Vectorizing through $\operatorname{vec}(XYZ)=(Z^\top\otimes X)\operatorname{vec}(Y)$, and using the commutation matrix $K_{d,p}$ defined by $\operatorname{vec}(X^\top)=K_{d,p}\operatorname{vec}(X)$ for $X\in\R^{d\times p}$, see \cite[Eq.~520]{petersen2012matrixcookbook}, gives
\begin{equation*}
J_i(A)=AN_i^\top\otimes I_d+(I_d\otimes AN_i)\,K_{d,p},
\qquad
J(A)=\begin{bmatrix}J_1(A)\\ J_2(A)\end{bmatrix}\in\R^{2d^2\times dp}.
\end{equation*}
Two rank-preserving operations simplify this Jacobian. For the first,  the invertibility of $K_{d,p}$ and $AN_i^\top\otimes I_d=K_{d,d}(I_d\otimes AN_i^\top)K_{d,p}$ permit right multiplication by $K_{d,p}^{-1}$, producing the rows $K_{d,d}(I_d\otimes AN_i^\top)+I_d\otimes AN_i$. Second, replacing the second row by the sum of both is an invertible row operation. Hence
\begin{equation}\label{Eq:two_blocks}
\operatorname{rank}J(A)=\operatorname{rank}
\begin{bmatrix}
K_{d,d}(I_d\otimes AN_1^\top)+I_d\otimes AN_1\\[4pt]
K_{d,d}\bigl(I_d\otimes A(N_1+N_2)^\top\bigr)+I_d\otimes A(N_1+N_2)
\end{bmatrix},
\end{equation}
whose first row involves only $N_1$ and whose second row involves $N_1+N_2$. We evaluate the two rows of \eqref{Eq:two_blocks} at
\begin{equation}\label{eq:witness}
A:=\mathbf{1}_d^\top\otimes I_d=(I_d,\dots,I_d)\in\R^{d\times d^2},
\end{equation}
the $d$-fold horizontal concatenation of $I_d$, where $\mathbf{1}_d\in\R^d$ is the all-ones vector.
By \Cref{lemma:sigma12odd}, the second row satisfies
\begin{equation*}
N_1+N_2=W\oplus 0^{\oplus(d^2-d)},
\qquad
W:=2\oplus\bigl(2H_2(-1)\bigr)^{\oplus \kappa},\quad \kappa=\tfrac{d-1}{2}.
\end{equation*}
As $N_1+N_2$ is supported on its leading $d\times d$ block, we obtain $A(N_1+N_2)=e_1^\top\otimes W$ and $A(N_1+N_2)^\top=e_1^\top\otimes W^\top$. Since $H_2(-1)^\top=-H_2(-1)$, transposition fixes the leading scalar $2$ of $W$ and negates each $2H_2(-1)$ block. Setting $\epsilon_1:=1$ and $\epsilon_a:=-1$ for $2\le a\le d$, this reads
$
We_a=\epsilon_a\,W^\top e_a$ for $1\le a\le d$.
Substituting $A(N_1+N_2)=e_1^\top\otimes W$ and $A(N_1+N_2)^\top=e_1^\top\otimes W^\top$ into the second row of \eqref{Eq:two_blocks} and setting $R:=I_d\otimes(e_1^\top\otimes W^\top)$, we obtain the factorization
\begin{equation*}
L\,R=K_{d,d}\bigl(I_d\otimes A(N_1+N_2)^\top\bigr)+I_d\otimes A(N_1+N_2),
\end{equation*}
where $L:\R^{d^2}\to\R^{d^2}$ is the linear map defined on the standard basis by
\begin{equation*}
L(e_a\otimes e_b):=e_b\otimes e_a+\epsilon_b\,(e_a\otimes e_b).
\end{equation*}
Because $W$ is invertible, $R$ has full row rank $d^2$, so the second row and $L$ have equal rank.

The operator $L$ preserves each line $\operatorname{span}_\R(e_a\otimes e_a)$ for $1\le a\le d$ and each plane $\operatorname{span}_\R(e_a\otimes e_b,\,e_b\otimes e_a)$ for $1\le a<b\le d$, and $\R^{d^2}$ decomposes as their direct sum. On subspace $\operatorname{span}_\R(e_a\otimes e_a)$, $L$ acts as multiplication by $1+\epsilon_a$, which vanishes precisely for $a\ge 2$. On $\operatorname{span}_\R(e_a\otimes e_b,\,e_b\otimes e_a)$, $L$ acts as a matrix multiplication of $\begin{psmallmatrix}\epsilon_b & 1\\ 1 & \epsilon_a\end{psmallmatrix}$, 
of determinant $\epsilon_a\epsilon_b-1$, which vanishes precisely when $\epsilon_a=\epsilon_b$, that is, for $2\le a<b$. Hence
\begin{equation}\label{eq:kerL}
\dim\ker L=(d-1)+\binom{d-1}{2}=\binom{d}{2},
\qquad
\operatorname{rank}L=\frac{d(d+1)}{2}.
\end{equation}
The first row of \eqref{Eq:two_blocks} is the Jacobian of the parametrization $A\mapsto AN_1A^\top$ associated with the id-signature component alone, whose variety is the object of \cite[Theorem 4.7]{LS24}. That result asserts that this variety attains the full ambient dimension $d^2$; its proof shows that at the point \eqref{eq:witness} the Jacobian already has full row rank $d^2$.

It remains to add the two ranks. As the first row has full row rank $d^2$, the rank of \eqref{Eq:two_blocks} equals $d^2$ plus the dimension of the image of the second row restricted to the kernel of the first. Writing $R=I_d\otimes\bigl[\,W^\top\ 0\ \cdots\ 0\,\bigr]$, the second row equals $L\,R$, and therefore vanishes on the subspace $\ker R$ of perturbations of \eqref{eq:witness} supported on the copies of $I_d$ of index $2,\dots,d$. 
The first row maps $\ker R$ onto $\R^{d^2}$, so $\ker R$ and the kernel of the first row together span $\R^{d^3}$. Since the second row vanishes on $\ker R$, its image on the kernel of the first row is the entire image of the second row, of dimension $d(d+1)/2$ by \eqref{eq:kerL}. Consequently
\begin{equation*}
\operatorname{rank}J(A)=d^2+\frac{d(d+1)}{2}=\frac{1}{2}d(3d+1),
\end{equation*}
which coincides with \eqref{eq:upperbound}. Thus the generic rank of $J$ equals $d(3d+1)/2=\dim\mathcal{U}_{d,2}$.
\end{proof}

\begin{remark}\label{rem:whyGammaTilde}
For $d$ even, set $\kappa:=\tfrac{d-2}{2}$ as in \Cref{lemma:sigma12even}
we obtain 
\begin{equation*}
W:=\tilde\Gamma_2+\tilde\Gamma_2^\top=\begin{pmatrix}2&0\\0&0\end{pmatrix}
\end{equation*}
 of rank $1$. Summing the blocks of $N_1$ and $N_2$ from the lemma, position by position, the terms
$(\tilde\Gamma_2\otimes H_2(-1))^{\oplus\kappa}$ and
$(\tilde\Gamma_2\otimes(-H_2(-1)))^{\oplus\kappa}$ cancel, as do both pairs
of $H_2(1)^{\oplus\kappa^2}$ and $(-H_2(1))^{\oplus\kappa^2}$ blocks. The remaining two positions combine, leaving
\begin{equation*}
N_1+N_2=(\tilde\Gamma_2\otimes W)\oplus (H_2(-1)\otimes W)^{\oplus\kappa}.
\end{equation*}
Since $\tilde\Gamma_2$ and $H_2(-1)$ are invertible and $W$ has rank $1$,
each Kronecker factor contributes rank $2$, so
$\operatorname{rank}(N_1+N_2)=2+2\kappa=d,
$
matching $\operatorname{rank}(W)=d$ in the odd case of
\Cref{lemma:sigma12odd}. The remainder of the argument, with the witness
$A$ adapted to the block structure of $N_1+N_2$, proceeds exactly as in the
proof of \Cref{thm:dimensionUniversalVariety} and yields the same value
$\operatorname{rank}J(A)=d(3d+1)/2$.
\end{remark}

\begin{theorem}[Algebraic Chen-Chow for $k=2$]\label{cor:algChenChow-k2}
Every generic $1\oplus v\oplus s\in\C\otimes \C^d\oplus\C^{d\times d\times 2}$, which satisfies the two-parameter shuffle relations
$$v_i\cdot v_j=s_{i,j,1}+s_{j,i,1}+s_{i,j,2}+s_{j,i,2}$$
 for all $1\leq i,j\leq d$, lies in the Zariski-closed image of the $2$-truncated signature restricted to piecewise bilinear membranes of order $(d,d)$.
\end{theorem}
\begin{proof}
By hypothesis, we have $1\oplus v\oplus s\in\mathcal{G}_{d,2}=\mathcal{V}(G_{d,2})$ at truncation level $2$. Consider the projection $\operatorname{pr}:\mathcal{G}_{d,2}\to \mathcal{U}_{d,2}$ from \Cref{prop:proojectionResult} that forgets the first-level coordinate. In particular $s=\operatorname{pr}(1\oplus v \oplus s) \in \mathcal{U}_{d,2}$.

By \Cref{thm:dimensionUniversalVariety,thm:polyVsPwLin}, we have $\mathcal{U}_{d,2}=\mathcal{P}_{d,2,(d,d)}=\mathcal{L}_{d,2,(d,d)}$. Since $1\oplus v \oplus s$ is generic, there exists $A_0\in \C^{d\times d^2}$ such that $s=A_0*\sigma^{(2)}(\mathsf{Mom}^{d,d})$. Setting $v_0:=A_0\mathbf{1}_d$, the truncated signature $1\oplus v_0\oplus s$ satisfies the shuffle relations and lies in the fiber $\operatorname{pr}^{-1}(s)$.
By genericity, $s$ lies in the dense open subset $\mathcal{U}_{d,2}^{\circ}\subseteq\mathcal{U}_{d,2}$ from \Cref{prop:proojectionResult} over which the projection is exactly two-to-one, with the two preimages related by the dilation action of the square roots of unity $\{\pm1\}$. Thus
$\operatorname{pr}^{-1}(s)=\{\,v_0\oplus s,\ -v_0\oplus s\,\}$, so either $v=v_0$ or $v=-v_0$.
If $v=v_0$, set $A:=A_0$.  If $v=-v_0$, set $A:=-A_0\in\C^{d\times d^2}$. The action $A*\sigma^{(\ell)}(\mathsf{Mom}^{d,d})$ is homogeneous of degree $\ell$ in $A$. Hence $(-A_0)*\sigma^{(\ell)}(\mathsf{Mom}^{d,d})=(-1)^\ell\,A_0*\sigma^{(\ell)}(\mathsf{Mom}^{d,d})$. In particular, $A\mathbf{1}_d=v$ and $A*\sigma^{(2)}$, so $1\oplus v \oplus s$ lies in the Zariski-closed image of the $2$-truncated signature restricted to piecewise bilinear membranes of order $(d,d)$, as desired.
\end{proof}

\begin{example}
For fixed $m$ and $n$, applying the same procedure as in \Cref{thm:dimensionUniversalVariety} but for the parametrization $A\mapsto\sigma^{(2)}(A*\mathsf{Mom}^{m,n})$ with $A\in\R^{d\times mn}$, we obtain the generic rank of its Jacobian $J_{d,m,n}$, whose value is  $\dim\mathcal{P}_{d,2,(m,n)}$ by \Cref{thm:polyVsPwLin}. In \Cref{fig:dimsMdmn}, we compute these dimensions with \cite{RS26} and present them in matrices indexed by $(m,n)$ until the full dimension $d(3d+1)/2$  has been reached.
\end{example}

\begin{table}[ht]
\centering
\begin{subtable}[t]{0.28\textwidth}\centering
$\begin{psmallmatrix} 3 \\ 5 & 10 \\ 6 & 13 & 15 \\ 6 & 14 & 15 & 15 \end{psmallmatrix}$
\caption*{$d=3$}
\end{subtable}\hfill
\begin{subtable}[t]{0.34\textwidth}\centering
$\begin{psmallmatrix} 4 \\ 7 & 14 \\ 9 & 19 & 24 \\ 10 & 22 & 25 & 26 \\ 10 & 23 & 25 & 26 & 26 \end{psmallmatrix}$
\caption*{$d=4$}
\end{subtable}\hfill
\begin{subtable}[t]{0.36\textwidth}\centering
$\begin{psmallmatrix} 5 \\ 9 & 18 \\ 12 & 25 & 33 \\ 14 & 30 & 36 & 38 \\ 15 & 33 & 37 & 39 & 40 \\ 15 & 34 & 37 & 39 & 40 & 40 \end{psmallmatrix}$
\caption*{$d=5$}
\end{subtable}

\medskip

\begin{subtable}[t]{0.28\textwidth}\centering
$\begin{psmallmatrix} 6 \\ 11 & 22 \\ 15 & 31 & 42 \\ 18 & 38 & 48 & 51 \\ 20 & 43 & 50 & 53 & 55 \\ 21 & 46 & 51 & 54 & 56 & 57 \\ 21 & 47 & 51 & 54 & 56 & 57 & 57 \end{psmallmatrix}$
\caption*{$d=6$}
\end{subtable}\hfill
\begin{subtable}[t]{0.34\textwidth}\centering
$\begin{psmallmatrix} 7 \\ 13 & 26 \\ 18 & 37 & 51 \\ 22 & 46 & 60 & 65 \\ 25 & 53 & 64 & 68 & 71 \\ 27 & 58 & 66 & 70 & 73 & 75 \\ 28 & 61 & 67 & 71 & 74 & 76 & 77 \\ 28 & 62 & 67 & 71 & 74 & 76 & 77 & 77 \end{psmallmatrix}$
\caption*{$d=7$}
\end{subtable}
\begin{subtable}[t]{0.36\textwidth}\centering
$\begin{psmallmatrix}
  8   \\
 15 & 30 \\
 21 & 43 & 60 \\
26 & 54 & 72 & 80 \\
30 & 63 & 79 & 84 & 88 \\
33 & 70 & 82 & 87 & 91 & 94 \\
35 & 75 & 84 & 89 & 93 & 96 & 98  \\
36 & 78 & 85 & 90 & 94 & 97 & 99 & 100  \\
36 & 79 & 85 & 90 & 94 & 97 & 99 & 100 & 100 \end{psmallmatrix}$
\caption*{$d=8$}
\end{subtable}

\caption{The dimension matrices $\left(\dim\mathcal{P}_{d,2,(m,n)}\right)_{m,n}$ for varying $d$.}\label{fig:dimsMdmn}
\end{table}

The pattern observed in \Cref{fig:dimsMdmn} admits a uniform closed form which we record below. 

\begin{conjecture}\label{conj:closedFormDimL}
For all $d,m,n\in\mathbb{N}$,
\begin{equation}\label{eq:closedFormDimL}
\dim\mathcal{L}_{d,2,(m,n)} = d\,mn - s_d(m) - s_d(n) - s_d\bigl((m-1)(n-1)\bigr),
\end{equation}
where
\begin{equation*}
s_d(m):=\begin{cases}\binom{m}{2} & \text{if $m\le d$,}\\[2pt] dm-\binom{d+1}{2} & \text{if $m\ge d$.}\end{cases}
\end{equation*}
\end{conjecture}

At $(m,n)$ large, so that $m,n\ge d$ and $(m-1)(n-1)\ge d$, we have $s_d(m)=dm-\binom{d+1}{2}$, $s_d(n)=dn-\binom{d+1}{2}$, and $s_d\bigl((m-1)(n-1)\bigr)=d(m-1)(n-1)-\binom{d+1}{2}$, so \eqref{eq:closedFormDimL} collapses to $d(3d+1)/2$.

\begin{remark}
    Using \cite{RS26}, we provide computational evidence for \Cref{conj:closedFormDimL} until $d\le 12$ and $1\le m,n\le d+1$. 
\end{remark} 

\section{Witt-type formula for two-parameter shuffle algebras}\label{sec:Witt}

In this section, we determine the dimension of the $k$th universal variety, by constructing a novel Witt-type formula (\Cref{thm:witt-two-parameter}). With this we count the number of free generators in the two-parameter shuffle algebra.

Let $H_\ell$ denote the $\ell$th graded component, with respect to weight, of the two-parameter
shuffle algebra $H=\bigoplus_{\ell\geq0}H_{\ell}$. 
Counting permutations gives the vector space dimension  $\dim_\R H_\ell=\ell!\,d^{\,\ell}$. Let $g_\ell$ denote the
number of free generators of $H$ of weight $\ell$ according to \Cref{cor:free}. Since $H$ is free, it is isomorphic to the symmetric algebra $\mathrm{Sym}(V)$ on its space of free
generators $V=\bigoplus_{\ell\geq1}V_\ell$, the graded polynomial ring generated by $V$, e.g., see \cite[Chapter A2]{eisenbud1995commutative}. Note that  $\dim_\R V_\ell=g_\ell<\infty$.

We now derive the Hilbert--Poincar\'e series of $H$ from this description. Since $\mathrm{Sym}$ carries direct sums to
graded tensor products, we decompose
\begin{equation*}
H\;\cong\;\bigotimes_{\ell\geq1}\mathrm{Sym}(V_\ell),
\end{equation*}
so it suffices to determine the Hilbert--Poincar\'e series of each factor separately and then multiply them. The factor $\mathrm{Sym}(V_\ell)$ is a polynomial ring in $g_\ell$ variables, each of degree $\ell$. Its degree-$n$ part is spanned by monomials of total degree $n$ in those $g_\ell$ variables, which is nonempty only when $\ell$ divides $n$. For $n=\ell j$ the number of such monomials is the number of degree-$j$ monomials in $g_\ell$ variables, namely $\binom{j+g_\ell-1}{g_\ell-1}$. Summing over $j\geq0$ and applying the negative binomial expansion gives the Hilbert--Poincar\'e series
\begin{equation*}
\sum_{j\geq0}\binom{j+g_\ell-1}{g_\ell-1}q^{\ell j}\;=\;(1-q^\ell)^{-g_\ell}
\end{equation*}
as an identity of formal power series in $\mathbb{Q}[\![q]\!]$.

Multiplying these factors over $\ell$ (whose one-parameter precedent for shuffle algebras
is Radford's theorem \cite[Theorem 3.1.1(e)]{radford1979natural}, see also
\cite[Theorem 6.3.4]{grinbergreiner2020hopf}), 
yields the identity
\begin{equation}\label{eq:hilbert-series}
  \sum_{\ell\geq0}\dim_\R (H_\ell)\,q^{\ell}
  \;=\;\prod_{\ell\geq1}(1-q^{\ell})^{-g_\ell}
\end{equation}
with $\dim_\R H_\ell=\ell!\,d^{\,\ell}$. Our goal is to solve \eqref{eq:hilbert-series} in closed form for all $\ell$
simultaneously, following the argument Witt used in the one-parameter case, e.g., compare 
\cite[Theorem~6.1]{reutenauer1993free}. For this, we first define a sequence that will
be of use in this endeavor.

\begin{definition}\label{def:B-sequence}
For $\ell\geq1$ let $B_\ell$ be the sequence recursively defined by $B_1:=1$ and
\begin{equation}\label{eq:B-recursion}
  B_\ell\;:=\;\ell\cdot\ell!\;-\;\sum_{k=1}^{\ell-1}k!\,B_{\ell-k}
\end{equation}
for $\ell\geq2$. Its first values are $1,\,3,\,13,\,71,\,461,\,3447,\,29093$ and $273343$.
The sequence $B_\ell$ counts the indecomposable (or connected) permutations of
$\{1,\dots,\ell+1\}$, those not fixing $\{1,\dots,j\}$ setwise for any
$0<j\leq\ell$, see \cite{lentin1972equations,oeisA003319}. 
\end{definition}

With $B_\ell$ in hand, we can state our main result of this section, the Witt-type formula for the two-parameter shuffle algebra $H$. 

\begin{theorem}\label{thm:witt-two-parameter}
The number of free weight-$\ell$ generators in the two-parameter shuffle algebra $H$ 
 is given by 
\begin{equation}\label{eq:witt-two-parameter}
  g_\ell=
        \frac{1}{\ell}\sum_{a\mid\ell}\mu\left(\frac{\ell}{a}\right)\,B_a\,d^{\,a}
\end{equation}
where $\mu$ denotes the M\"obius function and $B_a$ is as in
\Cref{def:B-sequence}.
\end{theorem}

\begin{example}\label{ex:gens-k3-witt}
With $B_1=1$, $B_2=3$, $B_3=13$ and $\mu(1)=1$,
$\mu(2)=\mu(3)=-1$, the theorem  gives
\begin{equation*}
  g_1=d,\qquad
  g_2=\frac{d(3d-1)}{2},\qquad
  g_3=\frac{d(13d^{2}-1)}{3}.
\end{equation*}
In particular, summing the first two terms we recover $\dim\mathcal{U}_{d,2}=g_1+g_2=d(3d+1)/2$ from \Cref{thm:dimensionUniversalVariety}.
\end{example}

Before proving \Cref{thm:witt-two-parameter}, we establish an auxiliary
lemma. 

\begin{lemma}\label{lem:factorial-ode}
Let $f(x):=\sum_{n\geq0}n!\,x^{n}\in 1+x\,\mathbb{Q}[\![x]\!]$. Then, with $B_\ell$ from \Cref{def:B-sequence}, 
\begin{enumerate}[(i)]
\item\label{lem:factorial-ode1} $x^{2}\frac{\partial}{\partial x}f(x)=(1-x)\,f(x)-1$,
\item\label{lem:factorial-ode2} $\displaystyle x\frac{\partial}{\partial x}\log f(x)=\sum_{\ell\geq 1}B_\ell\,x^\ell$.
\end{enumerate}
\end{lemma}

\begin{proof}
For part \eqref{lem:factorial-ode1},  we compare coefficients of $x^{\ell}$ on both sides. With the formal differential, $\frac{\partial}{\partial x}f(x)=\sum_{n\geq 1}n\cdot n!\,x^{n-1}$, the substitution $\ell=n+1$ gives 
\begin{equation*}
    x^{2}\dfrac{\partial }{\partial x}f(x)\;=\;\sum_{\ell\geq2}(\ell-1)\,(\ell-1)!\,x^{\ell},
\end{equation*}
with no contribution in degrees zero and one. On the right,
\begin{equation*}
    (1-x)f(x)-1=\sum_{\ell\geq1}\ell!\,x^{\ell}
    -\sum_{\ell\geq1}(\ell-1)!\,x^{\ell}
  =\sum_{\ell\geq1}\bigl(\ell!-(\ell-1)!\bigr)x^{\ell},
\end{equation*}
the constant terms of $f(x)$ and of the subtracted $1$ cancel each other out directly. Since $\ell!-(\ell-1)!=(\ell-1)(\ell-1)!$ vanishes at $\ell=1$, both series equal $(\ell-1)(\ell-1)!\,x^{\ell}$ for $\ell\geq2$ and $0$ otherwise. 

For part \eqref{lem:factorial-ode2}, $f(x)$ lies in $1+x\,\mathbb{Q}[\![x]\!]$ so its logarithm is defined. Let $b_\ell:=\ell\,[x^{\ell}]\log f(x)$, where $[x^\ell]$ denotes the dual of $x^\ell$, so that
\begin{equation*}
    x\frac{\partial}{\partial x}\log f(x)\;=\;\sum_{\ell\geq 1}b_\ell\,x^\ell.
\end{equation*}
We show $b_\ell=B_\ell$ for all $\ell\geq 1$. Since $f(x)\in 1+x\,\mathbb{Q}[\![x]\!]$ is invertible in $\mathbb{Q}[\![x]\!]$, dividing part \eqref{lem:factorial-ode1} by $f(x)$ gives
\begin{equation*}
  x\dfrac{\partial}{\partial x}\log f(x)\;=\;\frac{x\frac{\partial}{\partial x} f(x)}{f(x)}
        \;=\;\frac{1}{x}\Bigl(1-x-\frac{1}{f(x)}\Bigr),
\end{equation*}
equivalently
\begin{equation}\label{eq:phi-functional}
  f(x)\Bigl(1-x-\sum_{\ell\geq 1}b_\ell\,x^{\ell+1}\Bigr)\;=\;1 .
\end{equation}
Expanding the left-hand side of \eqref{eq:phi-functional} as a Cauchy product,
\begin{equation*}
  f(x)-xf(x)-\Bigl(\sum_{\ell\geq 1}b_\ell\,x^{\ell+1}\Bigr)f(x)
  \;=\;\sum_{n\geq0}n!\,x^{n}-\sum_{n\geq1}(n-1)!\,x^{n}
     -\sum_{n\geq2}\Bigl(\sum_{j=1}^{n-1}b_j\,(n-1-j)!\Bigr)x^{n}.
\end{equation*}
Degrees zero and one are automatic, for $n\geq 2$ the coefficient of $x^n$ must vanish, giving
\begin{equation*}
     n!-(n-1)!-\sum_{j=1}^{n-1}b_j\,(n-1-j)!\;=\;0 .
\end{equation*}
Set $\ell:=n-1 \geq 1$ and isolate the summand $j=\ell$. Using $(\ell+1)!-\ell!=\ell\cdot\ell!$ and the change of index $k=\ell-j$,
\begin{equation*}
  b_\ell\;=\;\ell\cdot\ell!-\sum_{k=1}^{\ell-1}k!\,b_{\ell-k}.
\end{equation*}
For $\ell=1$ the sum is empty and gives $b_1=1=B_1$, for $\ell\geq2$ this is precisely \eqref{eq:B-recursion}. Since the recursion determines its solution uniquely, $b_\ell=B_\ell$ for all $\ell\geq1$. 
\end{proof}

\begin{proof}[Proof of \Cref{thm:witt-two-parameter}]
Both sides of \eqref{eq:hilbert-series} lie in $1+q\,\mathbb{Q}[\![q]\!]$ and can be compared coefficient by coefficient, where $[q^\ell]$ denotes dual of $q^\ell$. We compute $[q^\ell]\log$ of each side.

For the right-hand side, expanding $-\log(1-q^m)=\sum_{j \geq 1} q^{mj}/j$ and grouping the finitely many pairs $(m,j)$ with $mj=\ell$ yields 
\begin{equation}\label{eq:log-rhs}
  [q^{\ell}]\log\prod_{m\geq1}(1-q^{m})^{-g_m}
  \;=\;\sum_{m\geq1}g_m\sum_{j\geq1}\frac{[q^\ell]q^{mj}}{j}
  \;=\;\frac{1}{\ell}\sum_{m\mid\ell}m\,g_m .
\end{equation}
For the left-hand side, $\dim_\R H_\ell=\ell!\,d^{\,\ell}$ gives $\sum_{\ell\geq0}\dim_\R (H_\ell)\,q^{\ell}=f(dq)$ with $f$ as in \Cref{lem:factorial-ode}. Changing variables $x=dq$, the chain rule and part \eqref{lem:factorial-ode2} of \Cref{lem:factorial-ode} give
\begin{equation*}
    q\frac{\partial}{\partial q}\log f(dq)\;=\;x\frac{\partial}{\partial x}\log f(x)\Big|_{x=dq}\;=\;\sum_{\ell\geq 1}B_\ell\,d^{\,\ell}\,q^{\ell},
\end{equation*}
so $\ell\,[q^\ell]\log f(dq)=B_\ell\,d^{\,\ell}$. Equating $[q^\ell]\log$ of both sides of \eqref{eq:hilbert-series} and multiplying by $\ell$ yields for all $\ell\geq1$, 
\begin{equation}\label{eq:witt-type-identity}
  \sum_{m\mid\ell}m\,g_m\;=\;B_\ell\,d^{\,\ell}.
\end{equation}
Setting $F(\ell):=B_\ell\,d^{\,\ell}$ and $G(m):=m\,g_m$, \eqref{eq:witt-type-identity} reads $F(\ell)=\sum_{m\mid\ell}G(m)$ for all $\ell\geq1$. By M\"obius inversion, e.g., see \cite[\S3.7]{stanley2011enumerative}, we obtain $G(\ell)=\sum_{a\mid\ell}\mu(\ell/a)F(a)$, and dividing by $\ell$ yields \eqref{eq:witt-two-parameter}.
\end{proof}

The Witt-type formula \eqref{eq:witt-two-parameter} counts the free generators
of the two-parameter shuffle algebra $H$ in each weight. Summing over
weights $1$ to $k$ yields the dimension of the universal
variety.

\begin{theorem}\label{thm:dimUdk}
For all $d,k\geq1$,
\begin{equation}\label{eq:dim-Udk-bound}
  \dim\mathcal{U}_{d,k}=\sum_{\ell=1}^{k}\sum_{a\mid\ell}\mu\left(\frac{\ell}{a}\right)\frac{B_a\,d^{a}}{\ell}. 
\end{equation}
\end{theorem}
\begin{proof}
  By \Cref{cor:free}, $H\cong \mathrm{Sym}(V)$ with $V=\bigoplus_{\ell \geq 1} V_{\ell}$. The ideal $G_{d,k}$ collects exactly the shuffle relations among signature coordinates of weight $k$ or lower, so the quotient $\R[\sigma^{(\leq k)}]/G_{d,k}$ is the $k$th truncation of $H$, isomorphic to $\mathrm{Sym}(V_{\leq k})$ with $V_{\leq k}:=\bigoplus_{\ell=1}^k V_\ell$.
  Since the symmetrization of a vector space is a polynomial ring on that space, $\mathcal{G}_{d,k}$ is an affine space with 
$\dim V_{\leq k}=\sum_{\ell=1}^k g_\ell$.  By \Cref{prop:proojectionResult},  $\mathcal{G}_{d,k}\to\mathcal{U}_{d,k}$ is generically $k$-to-one,
so generically finite onto its image, so
$\dim\mathcal{U}_{d,k}=\dim\mathcal{G}_{d,k}=\sum_{\ell=1}^k g_\ell$.
Substituting \eqref{eq:witt-two-parameter} gives the stated formula.
\end{proof}

\begin{table}[ht]
\centering
\begin{subtable}[t]{0.45\textwidth}\centering
$\begin{psmallmatrix} 2 \\ 4 & 8 \\ 5 & 12 & 18 \\ 5 & 16 & 24 & 32 \\ 5 & 20 & 30 & 40 & 41 \\ 5 & 24 & 36 & 41 & 41 & 41 \end{psmallmatrix}$
\caption*{$d=2$}
\end{subtable}\hfill
\begin{subtable}[t]{0.45\textwidth}\centering
$\begin{psmallmatrix} 3 \\ 6 & 12 \\ 9 & 18 & 27 \\ 12 & 24 & 36 & 48 \\ 14 & 30 & 45 & 60 & 75 \\ 14 & 36 & 54 & 72 & 90 & 108 \\ 14 & 42 & 63 & 84 & 105 & 126 & 131 \\ 14 & 48 & 72 & 96 & 120 & 131 & 131 & 131 \end{psmallmatrix}$
\caption*{$d=3$}
\end{subtable}
 
\medskip

\begin{subtable}[t]{0.45\textwidth}\centering
$\begin{psmallmatrix} 4 \\ 8 & 16 \\ 12 & 24 & 36 \\ 16 & 32 & 48 & 64 \\ 20 & 40 & 60 & 80 & 100 \\ 24 & 48 & 72 & 96 & 120 & 144 \\ 28 & 56 & 84 & 112 & 140 & 168 & 196 \\ 30 & 64 & 96 & 128 & 160 & 192 & 224 & 256 \\ 30 & 72 & 108 & 144 & 180 & 216 & 252 & 288 & 302 \end{psmallmatrix}$
\caption*{$d=4$}
\end{subtable}
\hfill
\begin{subtable}[t]{0.45\textwidth}\centering
$\begin{psmallmatrix} 
5 \\
10 & 20  \\
15 & 30 & 45  \\
20 & 40 & 60 & 80  \\
25 & 50 & 75 & 100 & 125  \\
30 & 60 & 90 & 120 & 150 & 180   \\
35 & 70 & 105 & 140 & 175 & 210 & 245  \\
40 & 80 & 120 & 160 & 200 & 240 & 280 & 320   \\
45 & 90 & 135 & 180 & 225 & 270 & 315 & 360 & 405  \\
50 & 100 & 150 & 200 & 250 & 300 & 350 & 400 & 450 & 500 \\
55 & 110 & 165 & 220 & 275 & 330  & 385 & 440 & 495 & 550 & 580\\
\end{psmallmatrix}$
\caption*{$d=5$}
\end{subtable}

\caption{The dimension matrices  $\left(\dim\mathcal{P}_{d,3,(m,n)}\right)_{m,n}$ for varying $d$.}\label{fig:dimsMdmn-k3}
\end{table}

\begin{example}\label{ex:witt-vs-lyndon}
By \cite[Theorem 3.12]{schmitz2024freegeneratorshoffmansisomorphism}, $H$ is
freely generated via a family of Lyndon
words over the graded alphabet $C_\mathsf{irr}$, so $g_\ell$ counts Lyndon
words of total weight $\ell$ over this alphabet.
For $d=2$ and $k=3$, \Cref{thm:witt-two-parameter} gives $g_1=2$, $g_2=5$ and
$g_3=34$, so $H$ has $41$ free generators of weight at most three. 
The corresponding Lyndon words for $H_1$ are $\sigma_{(\w{1},1)}$ and $\sigma_{(\w{2},1)}$. For $H_2$ we obtain five Lyndon words of weight two, that is $\sigma_{(\w{1},1)}\sigma_{(\w{2},1)}$ of length two, and four singletons $\sigma_{(\w{ij},2\,1)}$ with arbitrary letters $\mathbf{i},\mathbf{j}\in\{\w{1},\w{2}\}$. For $H_3$ we get $\sigma_{(\w{1},1)}\sigma_{(\w{1},1)}\sigma_{(\w{2},1)}$ and $\sigma_{(\w{1},1)}\sigma_{(\w{2},1)}\sigma_{(\w{2},1)}$ of length three, $8$ words $\sigma_{(\w{ij},2\,1)}\sigma_{(\w{m},1)}$ of length two, and $24$ singletons $\sigma_{(\w{ijm},\tau)}$ with $\tau\in\{(2,3,1),(3,1,2),(3,2,1)\}$ and $\w{i},\w{j},\w{m}\in\{\w{1},\w{2}\}$. 
\end{example}

\begin{example}\label{ex:dim_P2mn_d2to5}
We use \cite{RS26} and compute the
polynomial membrane varieties $\mathcal{P}_{d,3,(m,n)}$ for $d\in\{2,3,4,5\}$ and truncation level three, see 
\Cref{fig:dimsMdmn-k3} for its dimensions. In each case, the entries of the resulting dimension matrices grow with
$(m,n)$ until they plateau at dimension of the universal variety $\dim\mathcal{U}_{d,3}=d/6 + 3d^2/2 + 13d^3/3$ from \Cref{thm:dimUdk},
namely $41$, $131$, $302$, and $580$, respectively. In particular, this matches \Cref{ex:witt-vs-lyndon}.
\end{example}

In all our computations from \Cref{ex:dim_P2mn_d2to5}, the tables $\dim(\mathcal{P}_{d,k,(m,n)})
$ with varying $(m,n)$ plateau at exactly $\dim\mathcal{U}_{d,k}$ for  $(m,n)$
large enough. 
This provides computational evidence for the following conjecture on polynomial and piecewise bilinear membranes towards universal varieties.

\begin{conjecture}\label{conj:stabilizingGridUniversalVarity}
The grid \eqref{eq:grid_of_varieties} stabilizes in the $k$-level universal variety, i.e., 
\begin{equation*} 
  \mathcal{U}_{d,k}\;=\;\bigcup_{M,N\to\infty}\mathcal{P}_{d,k,(M,N)}\;=\;\bigcup_{M,N\to\infty}\mathcal{L}_{d,k,(M,N)}
\end{equation*}
and the limit is reached at $M=N=\left\lceil\sqrt{\frac{\dim\mathcal{U}_{d,k}}d}\;\right\rceil$ for all $k\geq 3$. 
\end{conjecture} 

\begin{example}\label{ex:weak-chen-does-not-stabilize}
Take $d=2$ and $k=3$. For $a\in\R^{2}$, let $X(a):[0,1]^{2}\to\R^{2}$ denote a bilinear membrane with $\partial_{12}X(a)=a$, and for each $\ell\geq1$ let $\mathcal{W}_{2,3,\ell}\subseteq\mathcal{U}_{2,3}$ denote the Zariski closure of the image of the parametrization by $\R^{2\times \ell}$ via
$(a_{1},\ldots,a_{\ell}) \mapsto\ \sigma\bigl(X(a_{1})\sqcup\cdots\sqcup X(a_{\ell})\bigr)$.
By our notion of Chen's identity (\Cref{thm:Chen}) iterated $\ell-1$ times, this parametrization is polynomial in the entries of $(a_{1},\ldots,a_{\ell})$. Using  \cite{RS26} we verify computationally that $\dim\mathcal{W}_{2,3,\ell}=2\ell$ for $1\leq\ell\leq 9$, 
after which it stabilizes at $w_\ell = 18$ from $\ell\geq9$ onwards. In particular, $\bigcup_{\ell\geq 1}\mathcal{W}_{2,3,\ell}$ has dimension $18$, strictly below $\dim\mathcal{U}_{2,3}=41$ from \Cref{thm:dimUdk} (see also \Cref{ex:witt-vs-lyndon}). Hence our $\sqcup$-decomposable membranes do not fill the universal variety $\mathcal{U}_{d,k}$, in contrast to the path case by \cite[Theorem 6.1]{bib:AFS2019}. This shows that a proof of \Cref{conj:stabilizingGridUniversalVarity} must go beyond our notion of Chen's identity for membranes alone.
\end{example}

Inspired by \Cref{cor:algChenChow-k2},  \Cref{conj:stabilizingGridUniversalVarity} and the Chen-Chow theorem for paths (\cite[Theorem 7.28]{Friz_Victoir_2010}), we conjecture the following holds for membranes. 

\begin{conjecture}[Chen-Chow for membranes]\label{con:algebraicChenChowAllTruncLevels}
  For every sequence  $s\in\bigoplus_{\ell=0}^k\bigoplus_{\nu\in\Sigma_\ell}(\R^{d})^{\otimes\ell}$ with entries satisfying the two-parameter shuffle identity (\Cref{thm:shuffle-identity}) there is a piecewise bilinear membrane such that its $k$-truncated signature coincides with $s$.
\end{conjecture}

\section*{Acknowledgements}
The authors acknowledge funding by the Deutsche
Forschungsgemeinschaft (DFG, German Research Foundation) – CRC/TRR 388
``Rough Analysis, Stochastic Dynamics and Related Fields'' – Project A04 and B01, 
516748464.

\printbibliography

\end{document}